\documentclass[journal,twoside,web]{ieeecolor}
\usepackage{generic}
\usepackage{cite}
\usepackage{amsmath,amssymb,amsfonts}
\usepackage{algorithmic}
\usepackage{graphicx}
\usepackage{algorithm,algorithmic}
\usepackage{textcomp}

\def\BibTeX{{\rm B\kern-.05em{\sc i\kern-.025em b}\kern-.08em
    T\kern-.1667em\lower.7ex\hbox{E}\kern-.125emX}}
\usepackage[english]{babel}

\usepackage{amsmath}
\usepackage{graphicx}
\usepackage[colorlinks=true, allcolors=blue]{hyperref}
\hypersetup{hidelinks=true}  

\usepackage{epstopdf}
\usepackage{amsmath} 
\usepackage{amssymb}  

\newcommand{\R}{\mathbb{R}}

\newcommand{\N}{\mathbb{N}}
\usepackage{mathtools}
\usepackage{arydshln}

\usepackage{xcolor}

\usepackage{comment}
\usepackage{algorithmic}
\usepackage{algorithm}

\usepackage{stfloats}

\makeatletter
\def\@begintheorem#1#2{%
  \@IEEEtmpitemindent\itemindent
  \topsep 0pt\rmfamily\trivlist%
  \item[]\textit{\indent #1\ #2:\ }\itemindent\@IEEEtmpitemindent
}
\def\@opargbegintheorem#1#2#3{%
  \@IEEEtmpitemindent\itemindent
  \topsep 0pt\rmfamily\trivlist%
  \item[]\textit{\indent #1\ #2\ (#3):\ }\itemindent\@IEEEtmpitemindent
}
\def\@endtheorem{\endtrivlist\unskip}
\makeatother

\newtheorem{definition}{Definition}
\newtheorem{assumption}{Assumption}
\newtheorem{lemma}{Lemma}
\newtheorem{proposition}{Proposition}
\newtheorem{theorem}{Theorem}
\newtheorem{corollary}{Corollary}
\newtheorem{remark}{Remark}
\newtheorem{example}{Example}

\usepackage[most]{tcolorbox}

\usepackage[prependcaption,colorinlistoftodos]{todonotes}

\usetikzlibrary{decorations.pathmorphing, arrows.meta, calc}

\newcommand{\subscr}[2]{#1_{\textup{#2}}}

\newcommand{\Vreg}{V_{\textup{reg}}}
\newcommand{\Freg}{F_{\textup{reg}}}
\newcommand{\ellreg}{\ell_{\textup{reg}}}

\newcommand{\diag}[1]{\mathrm{diag}\,[#1]}

\newcommand{\timeset}{\N}
\newcommand{\feedback}{k}
\newcommand{\nstate}{n}
\newcommand{\ninput}{m}

\allowdisplaybreaks

\begin{document}

\title{Regularized Model Predictive Control \\ via Contractivity and Implicit Lur'e Analysis}

\author{Ryotaro Shima, Anand Gokhale,  \IEEEmembership{Graduate Student Member, IEEE}, Alexander Davydov, \IEEEmembership{Member, IEEE}, and Francesco Bullo, \IEEEmembership{Fellow, IEEE}
\thanks{This work was supported in part by AFOSR award FA9550-22-1-0059.}
\thanks{Ryotaro Shima is with Toyota Central Research and Development Laboratories, Incorporated, 41-1, Yokomichi, Nagakute, Aichi, Japan
(e-mail: ryotaro.shima@mosk.tytlabs.co.jp).
This work was partially done while the author was at the Center for Control, Dynamical Systems, and Computation, UC Santa Barbara.}
\thanks{Anand Gokhale and Francesco Bullo are with the Center for Control, Dynamical Systems, and Computation, UC Santa Barbara, Santa Barbara, CA 93106 USA
(e-mail: anand\_gokhale@ucsb.edu; bullo@ucsb.edu).}
\thanks{Alexander Davydov was with the Center for Control, Dynamical Systems, and Computation, UC Santa Barbara, Santa Barbara, CA 93106 USA.
He is now with the Department of Mechanical Engineering, Rice University, Houston, TX 77005 USA
(e-mail: davydov@rice.edu).}}

\maketitle

\begin{abstract}
This paper develops a contraction-based stability analysis for \emph{regularized model predictive control} (MPC), whose feedback law is defined implicitly by a finite-horizon optimal control problem with an additional regularizing cost.
The proposed approach interprets regularized MPC as an \emph{implicit Lur'e system}, in which the regularizing cost perturbs the optimality conditions.
We develop a multiplier-based contraction framework for implicit Lur'e systems and derive linear matrix inequality conditions for regularized MPC with three broad classes of regularizers: convex smooth stage costs, convex closed proper stage costs, and differentiable regularizers with Lipschitz gradients.
Numerical studies on input and state soft penalties, hard input constraints, and sparsity-promoting penalties illustrate that regularization shapes closed-loop performance while retaining formal contraction-based stability guarantees.
\end{abstract}


\section{Introduction}

\subsection{Context and motivation}

Model Predictive Control (MPC) is a model-based control methodology in which the control input is determined by solving a finite-horizon optimal control problem.
An attractive feature of MPC is that its cost function can be designed flexibly to reflect performance objectives.
In particular, \emph{regularizing} costs are often introduced to encode specific control objectives, such as soft constraint enforcement \cite{NMCDO-LTB:94} or sparsity promotion \cite{MG:16, TA-HK-DA-GO-SL:16, TO-CRR:17, TD-WM-VI-IH-SV:21}.
Because MPC defines its feedback law implicitly through an optimization problem, such regularizers complicate the theoretical analysis of the closed-loop behavior.

Two analysis tools are central to this paper.
The first is the \emph{Lur'e form}---linear dynamics interconnected with nonlinear perturbation operators---which enables stability analysis via multipliers, that is, quadratic inequalities satisfied by the nonlinearities.
This tool has been applied to open-loop shaping of robust MPC \cite{SY-CM-HC-FA:13, JK-RS-MAM-FA:21}, whereas Lur'e analysis of \emph{closed-loop} MPC systems has received little attention.
The second tool is \emph{contraction theory}, which characterizes stability through \emph{incremental} analysis, i.e., by studying the difference between each pair of trajectories.
This incremental viewpoint yields a unified treatment of exponential stability, exponential convergence of time-varying systems, entrainment in periodically time-varying dynamics, robustness against disturbances and time delays, and modularity and interconnection properties \cite{FB:26-CTDS, AD-VC-AG-GR-FB:23f}.

This paper develops a stability framework for regularized MPC through an implicit Lur'e approach, enabling contractivity analysis of the closed-loop regularized MPC system.
We interpret the regularizer as a perturbation operator and derive linear matrix inequalities (LMIs) for closed-loop contractivity using incremental multipliers of that perturbation.
Establishing closed-loop contractivity is valuable precisely because of the unified guarantees noted above: it certifies that the closed loop is exponentially stable to a globally unique fixed point, that this fixed point is robust to disturbances arising from modeling error or stochasticity, and that the system remains well behaved under time-varying parameters---a property that is especially useful for time-varying setpoints.


\subsection{Literature review}

The classical stability theory of MPC has been developed primarily within a Lyapunov-based framework.
One representative approach achieves monotonic decrease of the optimal value function by using a control Lyapunov function as a terminal cost; see \cite{DQM-JBR-CVR-POMS:20} for thorough reviews.
This technique is also applied to sparsity-promoting MPC \cite{MG:16}.

In robust MPC, the contraction analysis of open-loop systems is conducted in \cite{SY-CM-HC-FA:13} to ensure the constraint satisfaction under disturbances.
Such approaches are often called tube MPC or constraint tightening, and have been extended to non-quadratic contractivity \cite{JK-RS-MAM-FA:21}.

From a practical perspective, MPC designers introduce regularizing costs to tailor closed-loop performance.
For instance, soft penalties that encourage constraint satisfaction have been used for chemical plants \cite{NMCDO-LTB:94} and quadrotors \cite{DAAL-DIAL-MSF-DKDV:25}, while sparsity-promoting regularization has been used for power management \cite{TA-HK-DA-GO-SL:16}, wafer control \cite{TO-CRR:17}, and satellite control \cite{TD-WM-VI-IH-SV:21}.

In robust control, Lur'e systems are analyzed by representing the quadratic constraints satisfied by the nonlinearities through multipliers, yielding contractivity conditions expressed as LMIs; see \cite[Chapter~8]{SB-LEG-EF-VB:94}.
The multiplier theory has been extended to contraction analysis; see \cite{MG-VA-ST-DA:23} for PI control and observer design, \cite{AL-AG-NH:20} for $\ell_1$ adaptive control, \cite{AD-FB:24i} for safety filters, \cite{MF-MM-GJP:20} for safety verification of neural networks,~\cite{AG-AVP-YK-FB:26a-journal} for control of neural networks.
This direction has been pushed further into \emph{implicit} Lur'e analysis \cite{BA-MC:11, LDA-MC:13}, which has also been applied to anti-windup control design \cite{IQ-ST-GV-LZ:22} and recurrent equilibrium networks \cite{MR-RW-IRM:24}.
The recent work \cite{SH-GB-FD-DLMP:26} draws on incremental quadratic constraints to analyze the stability of receding horizon games, yielding an approach closely related to implicit Lur'e analysis in a special case; see Appendix~\ref{appendix:comparison_multipliers} for a detailed comparison with one of our results.

To the best of the authors' knowledge, the present paper is the first to bring implicit Lur'e analysis into the MPC context and, in doing so, the first to establish contraction conditions for the \emph{closed-loop MPC system}.

\subsection{Contributions}

The first contribution of this paper is to reformulate closed-loop regularized MPC systems as implicit Lur'e systems, opening a new route to analyzing their closed-loop behavior.
Building on this reformulation, the second contribution is a contraction analysis of regularized MPC, which, compared with ordinary (Lyapunov) stability, delivers the unified guarantees discussed in Section~\ref{sec:preliminaries} \cite{FB:26-CTDS, AD-VC-AG-GR-FB:23f}.

Third, to enable this analysis, we develop a contractivity theory for implicit Lur'e systems.
We construct an incremental multiplier matrix for the solution mapping of a generalized equation and derive a matrix inequality that is sufficient for closed-loop contractivity, extending the technique of \cite{LDA-MC:13} to a more general setting (see Proposition~\ref{prop:implicit_lure} and Corollary~\ref{cor:implicit_lure_contraction}).
We further provide a sufficient condition for the feasibility of this matrix inequality (Proposition~\ref{prop:contractivity_feasibility}), as well as a sufficient condition for the uniqueness of solutions to the generalized equation (Lemma~\ref{lemma:wellposedness_subdifferential} and Proposition~\ref{prop:well-posedness_contraction}).

The fourth contribution is a contractivity analysis of regularized MPC for three classes of regularizers:
convex smooth regularizers (Theorem~\ref{thm:contractivity_lmi:convex_smooth}),
convex closed proper regularizers (Theorem~\ref{thm:contractivity_lmi:ccp}), and
Lipschitz-gradient regularizers (Theorem~\ref{thm:contractivity:lipschitz}), together with certificates for the feasibility of the proposed matrix inequality and the well-posedness of the optimal control problem (Theorem~\ref{thm:lipschitz:feasibility_wellposedness}); see Table~\ref{table:summary} for a summary.
These three classes encompass the following practical regularizers:
\begin{enumerate}
    \item soft penalties on state constraints,
    \item soft penalties on input constraints,
    \item hard constraints on inputs, and
    \item sparsity-promoting regularizers.
\end{enumerate}
Finally, for each of these regularizers we present a numerical case study demonstrating the performance improvements induced by regularization while preserving formal contraction-based stability guarantees.

\subsection{Organization}

The remainder of this paper is organized as follows.
Section~\ref{sec:preliminaries} recaps contraction theory and Lur'e analysis with incremental multipliers.
Section~\ref{sec:problem_setup} presents the regularized MPC problem setting and the three classes of regularizers considered in this paper.
Section~\ref{sec:implicit_lure} develops the methods used to analyze the contractivity of regularized MPC.
Section~\ref{sec:main_results} presents our main results on the contractivity of regularized MPC for the three classes of regularizers.
Section~\ref{sec:numerical_examples} provides numerical examples illustrating the effectiveness of the proposed framework.
Section~\ref{sec:conclusion} concludes the paper.

\subsection{Notation}

Let $\mathbb{R}$ denote the set of real numbers, with closed and open
intervals denoted by $[a,b]$ and $(a,b)$, respectively. The power set of a
set $A$ is $\mathcal{P}(A) \coloneqq \{B \mid B \subseteq A\}$, and the
direct sum of sets $A$ and $B$ is $A \oplus B$.  For vectors, $v_i$ denotes
the $i$-th element of $v$, and we define the weighted Euclidean norm as
$\|v\|_P \coloneqq \sqrt{v^\top P v}$, where $P$ is positive definite. We
concisely denote a column vector of scalars as $(a_1, \dots, a_n) \in
\mathbb{R}^n$, and the vertical concatenation of vectors $v^1, \dots, v^k$
as $[v^1; \dots; v^k] \coloneqq [(v^1)^\top \cdots\ (v^k)^\top]^\top$.  For
matrices, $I_n$ is the $n \times n$ identity matrix, $A \otimes B$ is the
Kronecker product, and $\mathrm{diag}(A_1, \dots, A_n)$ represents a
block-diagonal matrix with diagonal blocks $A_i$. The spectral radius of
$A$ is $\rho(A)$; $A$ is termed \emph{Schur} if $\rho(A) < 1$.  For a
scalar function $f \colon \mathbb{R}^n \to \mathbb{R}$, its gradient is
$\nabla f(x) \in \mathbb{R}^n$. The Jacobian of a vector-valued function $f
\colon \mathbb{R}^n \to \mathbb{R}^m$ is defined as $\frac{\partial
  f}{\partial x} \coloneqq \left[ \frac{\partial f}{\partial x_1} \ \cdots
  \ \frac{\partial f}{\partial x_n} \right] \in \mathbb{R}^{m \times n}$.

\section{Preliminaries}
\label{sec:preliminaries}

Throughout this paper, we consider the discrete-time linear time-invariant system
\begin{align}
    x(t+1) = Ax(t) + Bu(t),
    \label{plant_lti}
\end{align}
where $t \in \timeset$ denotes the time step, $x(t) \in \R^{\nstate}$ the state, $u(t) \in \R^{\ninput}$ the input, and $A \in \R^{\nstate \times \nstate}$, $B \in \R^{\nstate \times \ninput}$.

\subsection{Contraction Theory \cite{FB:26-CTDS}}

\begin{definition}[Contractivity]
    The discrete-time system $x(t+1) = f_t(x(t))$ is strongly contracting with respect to the norm $\|\cdot\|_P$ with factor $\eta \in (0,1)$ if, for any trajectories $\{x^{(1)}(t)\}$ and $\{x^{(2)}(t)\}$ of the closed-loop system and any time $t \in \timeset$,
    \begin{align}
        \label{eq:contraction_definition}
        \| x^{(1)}(t+1) - x^{(2)}(t+1)\|_{P}
        \le \eta \|x^{(1)}(t) - x^{(2)}(t)\|_{P}.
    \end{align}
\end{definition}

\smallskip

Let the (time-varying) feedback law for the system \eqref{plant_lti} be
\begin{align}
    \label{feedback_law}
    u(t) = \feedback_t(x(t)),
\end{align}
where $\feedback_t : \R^{\nstate} \to \R^{\ninput}$ for each $t \in \timeset$.

\begin{lemma}
\label{lemma:contraction_benefit}
Suppose that the system \eqref{plant_lti} with feedback law \eqref{feedback_law} is strongly contracting with factor $\eta$.
\begin{enumerate}
    \item \label{item:contraction:stability}
    All trajectories converge exponentially to the origin\footnote{
        A discrete-time system $x(t+1) = f_t(x(t))$ is said to converge to the origin exponentially with factor $\eta$ if there exists $c>0$ such that any trajectory $\{x(t)\}_{t \in \timeset}$ satisfies $\|x(t)\| \le c \eta^t \|x(0)\|$.
    } if $\feedback_t(0) \equiv 0$ for all $t \in \timeset$.
    \item \label{item:contraction:periodic}
    All trajectories converge to a periodic orbit with period $T$ if $\feedback_t(x)$ is periodic with respect to $t$ with period $T$.
\end{enumerate}
\end{lemma}

\proof
\ref{item:contraction:stability})
$\feedback_t(0) \equiv 0$ for all $t \in \timeset$ implies that the origin is an equilibrium of the closed-loop system.
The proof is completed by taking $\{x^{(1)}(t)\}$ arbitrarily and $x^{(2)}(t) \equiv 0$.

\ref{item:contraction:periodic})
The continuous-time version is presented in \cite[Theorem 3.15]{FB:26-CTDS}, and the proof is parallel to the discrete-time setting.
\endproof

\subsection{Lur'e Analysis for Contractivity}

\begin{definition}[Incremental multipliers for set-valued mappings]
\label{def:incremental_multiplier_set_valued}
Let $M \in \R^{(\nstate+r)\times(\nstate+r)}$ be symmetric.
A set-valued mapping $\Psi: \R^{\nstate} \to \mathcal{P}(\R^{r})$ admits $M$ as an incremental multiplier matrix (IMM) if, for all $z^{(1)}, z^{(2)} \in \R^{\nstate}$ and all $g^{(1)} \in \Psi(z^{(1)}),\, g^{(2)} \in \Psi(z^{(2)})$,
\begin{align}
    \label{incremental_multiplier_set_valued}
    \begin{bmatrix}
    z^{(1)} - z^{(2)} \\
    g^{(1)} - g^{(2)}
    \end{bmatrix}^\top
    M
    \begin{bmatrix}
    z^{(1)} - z^{(2)} \\
    g^{(1)} - g^{(2)}
    \end{bmatrix}
    \ge 0.
\end{align}
In particular, if $\Psi$ is a single-valued mapping, i.e., $\Psi: \R^{\nstate} \to \R^{r}$,\footnote{
    In this paper, a function $\psi: \R^{\nstate} \to \R^{r}$ is identified with the single-valued mapping $\Psi: \R^{\nstate} \to \mathcal{P}(\R^{r})$ defined by $\Psi(z) = \{\psi(z)\}$ for all $z \in \R^{\nstate}$.
} then necessarily $g^{(1)} = \Psi(z^{(1)})$ and $g^{(2)} = \Psi(z^{(2)})$ in \eqref{incremental_multiplier_set_valued}.
We refer to \eqref{incremental_multiplier_set_valued} as the IMM condition for $\Psi$.
\end{definition}

\begin{example}[Multiplier from monotonicity {\cite[\S 6.2]{LDA-MC:13}}]
\label{example:multiplier_monotone}
A set-valued mapping $\Psi: \R^{r} \to \mathcal{P}(\R^{r})$ is monotone if, for any $z^{(1)}, z^{(2)} \in \R^{r}$ and any $g^{(1)} \in \Psi(z^{(1)}), g^{(2)} \in \Psi(z^{(2)})$, we have $(z^{(1)} - z^{(2)})^\top (g^{(1)} - g^{(2)}) \ge 0$.
In this case, $\Psi$ admits
$
    \begin{bmatrix}
    0 & I_r \\
    I_r & 0
    \end{bmatrix}
$
as an IMM.
\end{example}

\begin{example}[Multiplier from Lipschitzness {\cite[\S 6.1]{LDA-MC:13}}]
\label{example:multiplier_lipschitz}
Let $\psi: \R^n \to \R^m$ be Lipschitz with Lipschitz constant $L$, i.e., $\|\psi(u) - \psi(v)\| \le L \|u-v\|$ holds for all $u, v \in \R^n$.
Then $\psi$ admits $M=\diag{I_n, -\frac{1}{L^2} I_m}$ as an IMM.
\end{example}

\begin{example}[Multiplier from convex analysis {\cite[Theorem 18.15]{HHB-PLC:17}, \cite[Theorem 2.1.5]{YN:18}}]
\label{example:multiplier_convex_analysis}
A function $f: \R^n \to \R$ is $L$-smooth if its gradient is Lipschitz with Lipschitz constant $L$.
If $f$ is convex as well, then we have $(\nabla f(u) - \nabla f(v))^\top (u-v) \ge \frac{1}{L} \|\nabla f(u) - \nabla f(v)\|^2$ for all $u, v \in \R^n$, which means that $\nabla f: \R^n \to \R^n$ admits $M=\begin{bmatrix} 0 & I_n \\ I_n & -\frac{2}{L} I_n \end{bmatrix}$ as an IMM.
\end{example}

\begin{lemma}[Contractivity via IMM]
\label{lemma:contraction_multiplier}
Consider the system \eqref{plant_lti} with feedback control \eqref{feedback_law}.
Let $P \in \R^{\nstate \times \nstate}$ be positive definite and $\eta \in (0,1)$.
Suppose that $\feedback_t$ in \eqref{feedback_law} admits $M_1, \ldots, M_p \in \R^{(\nstate+\ninput)\times(\nstate+\ninput)}$ as IMMs.
If there exist nonnegative scalars $\lambda_1, \ldots, \lambda_p$ satisfying
\begin{align}
    \label{contractivity_lmi_discrete}
    \begin{bmatrix}
        A^\top P A - \eta^2 P & A^\top P B \\
        B^\top P A & B^\top P B
    \end{bmatrix}
    +
    \sum_{i=1}^{p} \lambda_i M_i
    \preceq 0
    ,
\end{align}
then the closed-loop system is strongly contracting with respect to the norm $\|\cdot\|_P$ with factor $\eta$.
\end{lemma}

\proof
Let $x^{(1)}, x^{(2)} \in \R^{\nstate}$ be  arbitrary states, and let $Y \coloneqq [x^{(1)} - x^{(2)}; \, \feedback_t(x^{(1)}) - \feedback_t(x^{(2)})]$.
Upon pre-multiplying \eqref{contractivity_lmi_discrete} by $Y^\top$, post-multiplying by $Y$, and utilizing the fact that $Y^\top M_i Y \geq 0 \ \forall i = 1, \ldots, p$, we obtain
\begin{align}
     Y^\top
    \begin{bmatrix}
        A^\top P A - \eta^2 P & A^\top P B \\
        B^\top P A & B^\top P B
    \end{bmatrix}
    Y
    \le 0,
\end{align}
which in turn yields the contractivity condition~\eqref{eq:contraction_definition}.
\endproof

\begin{remark}
    A continuous-time counterpart for Lemma~\ref{lemma:contraction_multiplier} with a single matrix multiplier is presented in~\cite[Theorem 4.2]{LDA-MC:13}.
\end{remark}

\section{Problem Setup: Regularized MPC}

\label{sec:problem_setup}

Given the system \eqref{plant_lti} defined by $A \in \R^{\nstate \times \nstate}$ and $B \in \R^{\nstate \times \ninput}$, the \emph{regularized MPC} solves the following optimal control problem (OCP) at each time step:
\begin{subequations}
\label{ocp}
\begin{align}
    \label{argmin}
    & \hspace{-60pt}
    \begin{aligned}
        \min_{U \in \R^{H\ninput}}
        & \frac12 \sum_{h=1}^{H} (\| x_{h} \|_{Q}^2 + \|u_{h}\|_{R}^2)
        + \frac12 \|x_{H+1}\|_{Q_\mathrm{f}}^2
        \\
        & \hspace{100pt}
        + \Vreg(X,U,t)
    \end{aligned}
    \\
    \label{ocp_plant_lti}
    \mathrm{where} \quad
    x_{h+1} &= A x_h + B u_h, \quad h = 1, \ldots, H
    ,
    \\
    \label{ocp_initial_state}
    x_1 &= x
    ,
    \\
    U &\coloneqq [u_1; \ldots ; u_{H}] \in \R^{H\ninput}
    ,
    \\
    X &\coloneqq [x_2; \ldots ; x_{H+1}] \in \R^{H\nstate}
    .
\end{align}
\end{subequations}
Here, $H \in \N$ is the prediction horizon,
$u_1, \ldots, u_H$ form the input sequence over the horizon, and
$x_1, \ldots, x_{H+1}$ form the state sequence over the horizon with $x_1=x$.
Moreover, $Q \in \R^{n \times n}$, $R \in \R^{m \times m}$, and $Q_{\mathrm{f}} \in \R^{n \times n}$ are the state, input, and terminal cost matrices, respectively.
We refer to the cost term $\Vreg(X,U,t) \in \R \cup \{+\infty\}$ as the \emph{regularizer} (or \emph{regularizing cost}).
\footnote{
    The regularizer may be time-varying when, for instance, the cost weights vary over time or set points are specified along the prediction horizon.
    See also \S \ref{sec:numerical_example_hard_input_penalty} for an example of tracking MPC.
}
For brevity, when the regularizer does not depend on some of its arguments, we omit them; for instance, we write $\Vreg(U,t)$, $\Vreg(X,U)$, or $\Vreg(U)$ when $\Vreg$ is independent of $X$ and/or $t$.
The control law of the regularized MPC is
\begin{align}
    \label{mpc_control_law}
    u(t) = \Pi_1 U^\ast_t(x(t))
    ,
\end{align}
where $\Pi_1 = [I_{\ninput} \ 0] \in \R^{\ninput \times H\ninput}$ denotes the selection matrix that extracts $u_1$ from $U$ and $U^\ast_t(x)$ denotes the optimal solution to the OCP \eqref{ocp}.

By substituting \eqref{ocp_plant_lti} and \eqref{ocp_initial_state} into \eqref{argmin} and omitting the terms independent of $U$, we can rewrite the OCP \eqref{ocp} as
\begin{align}
    \label{ocp_expanded}
    \min_{U \in \R^{H\ninput}}
    U^\top C x + \frac12 U^\top D U
    + \Vreg(\bar{A}x + \bar{B}U, U, t)
    ,
\end{align}
where $\bar{A}, \bar{B}, C, D$ are defined as
\begin{gather*}
    \bar{A}
    \coloneqq
    \begin{bmatrix}
        A
        \\
        \vdots
        \\
        A^H
    \end{bmatrix}
    ,
    \quad
    \bar{B}
    \coloneqq
    \begin{bmatrix}
        B & 0 & \cdots & 0
        \\
        AB & B & \ddots & \vdots
        \\
        \vdots & \ddots & \ddots & 0
        \\
        A^{H-1}B & \cdots & AB & B
    \end{bmatrix}
    ,
    \\
    C \coloneqq \bar{B}^\top \bar{Q} \bar{A},
    \quad
    D \coloneqq \bar{R} + \bar{B}^\top \bar{Q} \bar{B}
    ,
    \\
    \begin{aligned}
    \bar{Q} &\coloneqq \diag{Q, \ldots, Q, Q_\mathrm{f}} \in \R^{H\nstate \times H\nstate}
    ,
    \\
    \bar{R} &\coloneqq \diag{R, \ldots, R} \in \R^{H\ninput \times H\ninput}
    .
    \end{aligned}
\end{gather*}
Here, $\bar{Q}$ has $H-1$ copies of $Q$ followed by $Q_\mathrm{f}$.

\begin{assumption}
\label{assumption:positive_definite_cost}
The matrix $D$ is positive definite, i.e., $D\succ 0$.
\end{assumption}

Assumption~\ref{assumption:positive_definite_cost} holds if $Q_\mathrm{f} \succeq 0$, $Q \succeq 0$, and $R \succ 0$.

We refer to the MPC with $\Vreg \equiv 0$ as the \emph{nominal MPC}.
The control law of the nominal MPC is given by $u(t) = - \Pi_1 D^{-1} C x(t)$, and the closed-loop system is given by
\begin{gather}
    \notag
    x(t+1) = A_\mathrm{nom} x(t),
    \\
    \label{A_nom}
    A_\mathrm{nom} \coloneqq A - B \Pi_1 D^{-1} C.
\end{gather}
The matrix $A_\mathrm{nom} \in \R^{\nstate \times \nstate}$ is Schur if and only if the closed-loop system of the nominal MPC is stable, which is the case when, for example, the terminal cost matrix $Q_\mathrm{f}$ is chosen as the solution to the discrete algebraic Riccati equation (DARE):
\begin{align}
    \label{dare}
    Q_\mathrm{f} = A^\top Q_\mathrm{f} A - A^\top Q_\mathrm{f} B (R + B^\top Q_\mathrm{f} B)^{-1} B^\top Q_\mathrm{f} A + Q
    .
\end{align}
See also \cite[\S 2.5.1]{JBR-DQM-MMD:19} for the stability of the nominal MPC.

The following sections introduce the three classes of regularizers considered in this paper, together with practical examples.

\subsection{Case 1: Convex smooth regularizers}

\begin{assumption}[Convex smooth regularizer]
\label{assumption:convex_smooth_regularizer}
The regularizer $\Vreg$ is of the form
\begin{align}
    \label{Vreg:input_stage_costs}
    \Vreg(U, t) = \sum_{h=1}^H \ellreg(u_h, t+h)
    ,
\end{align}
where $\ellreg(\cdot,t)$ is convex and $L$-smooth for each $t \in \timeset$.
\end{assumption}

\begin{example}[Soft penalty for input constraints]
\label{example:input_soft_penalty}
To encourage the inputs to satisfy $\underline{u} \leq u_i \leq \bar{u}$ $(i=1,\dots,\ninput)$, we introduce a soft penalty on constraint violations.
First, for a box constraint $ - a_\mathrm{th} \leq z \leq a_\mathrm{th}$ for scalar variable $z$ with threshold $a_\mathrm{th}>0$, we define the violation function
\begin{align}
    \label{penalty_function}
    f_\mathrm{p} (z, a_\mathrm{th})
    &=
    \begin{cases}
    \frac12 (z - a_\mathrm{th})^2
    & \mathrm{if}\ z > a_\mathrm{th}
    ,
    \\
    \frac12 (z + a_\mathrm{th})^2
    & \mathrm{if}\ z < - a_\mathrm{th}
    ,
    \\
    0
    & \mathrm{if}\ |z| \le a_\mathrm{th}
    .
    \end{cases}
\end{align}
With respect to $z$, the function $f_\mathrm{p}$ is convex and its gradient has Lipschitz constant at most $1$.
Using this function, we introduce the following quadratic input penalty as a regularizer:
\begin{align}
\label{input_soft_penalty}
\Vreg(U)
&= \sum_{h=1}^H \ellreg(u_h) ,
\\
\label{ellreg:input_soft_penalty}
\ellreg(u)
&= L \sum_{i=1}^\ninput f_\mathrm{p}(u_i - u_\mathrm{c}, a_\mathrm{th}) ,
\end{align}
where $a_\mathrm{th} \coloneqq \frac12(\bar{u} - \underline{u})$ and $u_\mathrm{c} \coloneqq \frac12(\bar{u} + \underline{u})$.
Since $f_\mathrm{p}$ is convex and $1$-smooth, $\ellreg$ is convex and $L$-smooth, and Assumption~\ref{assumption:convex_smooth_regularizer} holds for $\Vreg$ in \eqref{input_soft_penalty}.
\end{example}

\subsection{Case 2: Convex closed proper regularizers}

\begin{assumption}[Convex closed proper regularizer]
\label{assumption:convex_closed_proper_regularizer}
The regularizer $\Vreg$ is of the form \eqref{Vreg:input_stage_costs}, where $\ellreg(\cdot, t)$ is convex closed and proper (CCP) for each $t \in \timeset$.
\end{assumption}

\begin{example}[Hard constraints on inputs]
\label{example:input_hard_constraint}
The regularized MPC with hard input constraints satisfies Assumption~\ref{assumption:convex_closed_proper_regularizer} if the feasible set $\mathcal{U}_t \subseteq \R^{\ninput}$ is convex, closed, and nonempty for each $t \in \timeset$.
Indeed, the restriction of $u$ to $\mathcal{U}_t$ at each time $t$ can be achieved by choosing the regularizer as the indicator function of $\mathcal{U}_t$, i.e.,
\begin{align}
    \label{indicator_regularization}
    \ellreg(u, t) =
    \begin{cases}
        0, & u \in \mathcal{U}_t
        ,
        \\
        +\infty, & u \notin \mathcal{U}_t
        ,
    \end{cases}
\end{align}
which is CCP for each $t \in \timeset$ if $\mathcal{U}_t$ is convex, closed, and nonempty.
\end{example}

\begin{example}[Sparse control]
\label{example:sparse}
A common regularizer to promote sparsity of the input $u$ is the $\ell_1$ norm \cite{TA-HK-DA-GO-SL:16,TO-CRR:17}:
\begin{align}
\label{l1_regularization}
\ellreg(u) = \lambda \|u\|_1 ,
\end{align}
where $\lambda > 0$ is a tuning parameter that controls the sparsity.
Since $\ellreg$ in \eqref{l1_regularization} is CCP, Assumption~\ref{assumption:convex_closed_proper_regularizer} holds for $\Vreg$ in \eqref{l1_regularization}.
\end{example}

\subsection{Case 3: Lipschitz-gradient regularizers}

\begin{assumption}[Well-posedness of OCP]
    \label{assumption:lipschitz:well-posedness}
    For any $x$ and $t$, the OCP \eqref{ocp} (or \eqref{ocp_expanded}) admits a unique optimal solution and no other stationary points.\footnote{
        Cases~1 and 2 do not require this well-posedness assumption, since the objective functions are strongly convex under Assumption~\ref{assumption:positive_definite_cost}.
    }
\end{assumption}

We will provide a sufficient condition for Assumption~\ref{assumption:lipschitz:well-posedness} in Theorem~\ref{thm:lipschitz:feasibility_wellposedness} below.

Suppose $\Vreg$ is differentiable with respect to $(X,U)$.
Under Assumption~\ref{assumption:lipschitz:well-posedness}, the optimal input of the OCP \eqref{ocp} (or \eqref{ocp_expanded}) is given by the unique solution to the following first-order optimality equation:
\begin{align}
    \label{equation:implicit}
    Cx + DU + \psi_t(\bar{A}x + \bar{B}U, U) = 0
    ,
    \\
    \label{psi_gradient_lipschitz}
    \psi_t(X, U)
    \coloneqq
    \nabla_U \Vreg(X, U, t)
    +
    \bar{B}^\top \nabla_X \Vreg(X, U, t)
    .
\end{align}
Note that the function $\psi_t(\bar{A}x + \bar{B}U, U)$ plays the role of the gradient of the regularizer with respect to $U$.
Throughout this paper, we refer to $\psi_t$ as the \emph{gradient} of the regularizer.

\begin{assumption}[Lipschitz-gradient regularizer]
\label{assumption:gradient_lipschitz_regularizer}
The regularizer $\Vreg$ is differentiable with respect to $(X,U)$, and its gradient $\psi_t(X, U)$ defined in \eqref{psi_gradient_lipschitz} is Lipschitz with respect to $(X,U)$ uniformly in $t$, with Lipschitz constant $L$.
\end{assumption}

Assumption~\ref{assumption:gradient_lipschitz_regularizer} allows nonconvex regularizers, as illustrated by the following example.

\begin{example}[Soft penalty for state constraints]
\label{example:state_soft_penalty}
To encourage states to satisfy box constraints of the form $\underline{x} \leq x_i \leq \bar{x}$ $(i=1,\dots,\nstate)$, consider the following regularizer:
\begin{align}
    \label{Vreg:state_soft_penalty}
    \Vreg(X, U)
    &= \Freg(x_{H+1}) + \sum_{h=1}^{H} \ellreg(x_h, u_h)
    ,
    \\
    \label{ellreg_state_soft_penalty}
    \ellreg (x, u)
    &= \frac{L}{\sqrt{2}}
    \Big(
        \sum_{i=1}^n \frac{f_\mathrm{p}(x_i - x_\mathrm{c}, a_\mathrm{th})}{\|\bar{B}^\top\|} - \frac{\|u\|^2}{2}
    \Big)
    ,
    \\
    \label{Freg_state_soft_penalty}
    \Freg(x)
    &= \frac{L}{\sqrt{2}} \sum_{i=1}^n \frac{f_\mathrm{p}(x_i - x_\mathrm{c}, a_\mathrm{th})}{\|\bar{B}^\top\|}
    ,
\end{align}
where $a_\mathrm{th} = \frac12(\bar{x} - \underline{x})$, $x_\mathrm{c} = \frac12(\bar{x} + \underline{x})$, and $f_\mathrm{p}$ is defined in \eqref{penalty_function}.
In this setting, the gradient $\psi$ is given by
\begin{align}
    \notag
    \psi(X, U) &= - \frac{L}{\sqrt{2}} U + \psi_\mathrm{p}(X)
    ,
    \\
    \notag
    \psi_\mathrm{p}(X)
    &\coloneqq
    \frac{L}{\sqrt{2}\|\bar{B}^\top\|} \bar{B}^\top
    \begin{bmatrix}
    \frac{\partial f_\mathrm{p}}{\partial a}(X_1 - x_\mathrm{c}, a_\mathrm{th}) \\
    \vdots \\
    \frac{\partial f_\mathrm{p}}{\partial a}(X_{H\nstate} - x_\mathrm{c}, a_\mathrm{th})
    \end{bmatrix}
    .
\end{align}
As mentioned in Example~\ref{example:input_soft_penalty}, $\frac{\partial f_\mathrm{p}}{\partial a}$ is Lipschitz with respect to $a$ with Lipschitz constant at most $1$.
Thus, $\psi_\mathrm{p}$ has a Lipschitz constant at most $L/\sqrt{2}$.
For all $X, X^\prime \in \R^{H\nstate}$ and $U, U^\prime \in \R^{Hm}$, we have
\begin{align*}
    &
    \|\psi(X, U) - \psi(X^\prime, U^\prime)\|
    \\
    &\leq
    \|\psi_\mathrm{p}(X) - \psi_\mathrm{p}(X^\prime)\| + \frac{L}{\sqrt{2}} \|U - U^\prime\|
    \\
    &\leq \frac{L}{\sqrt{2}} (\|X - X^\prime\| + \|U - U^\prime\|)
    \\
    &\leq L \|[X; U] - [X^\prime; U^\prime]\|
    ,
\end{align*}
where the first inequality is from the triangle inequality, the second inequality is from the Lipschitzness of $\psi_\mathrm{p}$, and the last inequality is from Cauchy--Schwarz inequality.
Therefore, the overall mapping $\psi$ is Lipschitz with Lipschitz constant at most $L$, and the regularizer in \eqref{Vreg:state_soft_penalty}--\eqref{Freg_state_soft_penalty} satisfies Assumption~\ref{assumption:gradient_lipschitz_regularizer}.
\end{example}

\section{Methods: Implicit Lur'e Analysis}
\label{sec:implicit_lure}

Our key observation in the contractivity analysis of regularized MPC is that the gradient of the regularizer acts as a perturbation of the optimality conditions of the OCP \eqref{ocp}.
This viewpoint naturally motivates us to consider the \emph{implicit Lur'e system}:
\begin{subequations}
\label{implicit_lure_system}
\begin{gather}
    \label{implicit_lure_plant}
    x(t+1) = A x(t) + B u(t)
    ,
    \\
    \label{generalized_equation}
    0 \in Cx(t) + Du(t) + \Psi_t(Fx(t) + Gu(t))
    ,
\end{gather}
\end{subequations}
where $x \in \R^{\nstate}$ is the state and $u \in \R^{\ninput}$ is the input.
The input $u(t)$ at time $t$ is implicitly determined by the inclusion relation \eqref{generalized_equation} (which is referred to as the \emph{generalized equation} \cite[Chapter 2]{ALD-RTR:09}) with $C \in \R^{\ninput \times \nstate}$, $D \in \R^{\ninput \times \ninput}$, $F \in \R^{r \times \nstate}$, $G \in \R^{r \times \ninput}$ and $\Psi_t: \R^{r} \to \mathcal{P}(\R^{\ninput})$.
In regularized MPC, the generalized equation \eqref{generalized_equation} plays the role of the first-order optimality conditions of the OCP \eqref{ocp}, and $\Psi_t: \R^{r} \to \mathcal{P}(\R^{\ninput})$ plays the role of the gradient of the regularizer $\Vreg(X,U,t)$.

\subsection{Contraction Theory via Implicit Lur'e Analysis}

To facilitate theoretical analysis, we impose the following assumption on well-posedness of \eqref{generalized_equation}.

\begin{assumption}[Well-posedness of generalized equation]
\label{assumption:generalized_equation}
The generalized equation \eqref{generalized_equation} admits a unique solution for each $x$ and $t$.
\end{assumption}

We will provide sufficient conditions for Assumption~\ref{assumption:generalized_equation} in \S\ref{sec:well_posedness}.

\begin{proposition}[IMM for solutions to generalized equations]
\label{prop:implicit_lure}
Under Assumption~\ref{assumption:generalized_equation}, consider the generalized equation \eqref{generalized_equation} and denote its solution mapping by $u^\ast_t : \R^{\nstate} \to \R^{\ninput}$.
If $\Psi_t$ in \eqref{generalized_equation} admits the IMM $M \in \R^{(r+\ninput)\times(r+\ninput)}$, then the solution mapping $u^\ast_t$ admits the IMM
\begin{align}
    \label{multiplier_imp}
    \mathtt{sol}(M)
    \coloneqq
    \begin{bmatrix}
    F & G
    \\
    -C & -D
    \end{bmatrix}^\top
    M
    \begin{bmatrix}
    F & G
    \\
    -C & -D
    \end{bmatrix}
    .
\end{align}
\end{proposition}

\smallskip

\proof
Let $x^{(1)}, x^{(2)} \in \R^{\nstate}$ be arbitrary and let $Y \coloneqq [x^{(1)} - x^{(2)}; u^\ast_t(x^{(1)}) - u^\ast_t(x^{(2)})]$ and $z_\ast^{(i)} \coloneqq Fx^{(i)} + Gu^\ast_t(x^{(i)})$ for $i=1,2$.
Then, by \eqref{generalized_equation}, there exist $g_\ast^{(i)} \in \Psi_t(z_\ast^{(i)})$ for each $i=1,2$ such that $Cx^{(i)} + Du^\ast_t(x^{(i)}) + g_\ast^{(i)} = 0$, which yields
\begin{align*}
    \begin{bmatrix}
        z_\ast^{(1)} - z_\ast^{(2)} \\ g_\ast^{(1)} - g_\ast^{(2)}
    \end{bmatrix}
    =
    \begin{bmatrix}
        F & G
        \\
        -C & -D
    \end{bmatrix}
    Y
    .
\end{align*}
Upon taking $z^{(i)} = z_\ast^{(i)}$ and $g^{(i)} = g_\ast^{(i)}$ in \eqref{incremental_multiplier_set_valued}, we obtain
\begin{align*}
    Y^\top
    \begin{bmatrix}
        F & G
        \\
        -C & -D
    \end{bmatrix}^\top
    M
    \begin{bmatrix}
        F & G
        \\
        -C & -D
    \end{bmatrix}
    Y
    \ge 0
    ,
\end{align*}
which is exactly the IMM condition for $u^\ast_t$ with $\mathtt{sol}(M)$ defined in \eqref{multiplier_imp}.
\endproof

\begin{corollary}[Contractivity via implicit Lur'e analysis]
\label{cor:implicit_lure_contraction}
Consider the implicit Lur'e system \eqref{implicit_lure_system} with Assumption~\ref{assumption:generalized_equation}.
Suppose that $\Psi_t$ in \eqref{generalized_equation} admits $M_1, \dots, M_p \in \R^{(r+\ninput)\times(r+\ninput)}$ as IMMs.
Let $P \in \R^{\nstate \times \nstate}$ be positive definite and $\eta \in (0,1)$.
If there exist nonnegative scalars $\lambda_1, \dots, \lambda_p$ satisfying
\begin{align}
    \label{contractivity_lmi_implicit}
    \begin{aligned}
    &
    \begin{bmatrix}
        A^\top P A - \eta^2 P & A^\top P B \\
        B^\top P A & B^\top P B
    \end{bmatrix}
    +
    \mathtt{sol} \Big( \sum_{i=1}^p \lambda_i M_i \Big)
    \preceq 0
    ,
    \end{aligned}
\end{align}
where $\mathtt{sol}(\cdot)$ is defined as in \eqref{multiplier_imp},
then the closed-loop system is strongly contracting with respect to the norm $\|\cdot\|_P$ with factor $\eta$.
\end{corollary}

\proof
The statement follows directly from Proposition~\ref{prop:implicit_lure} and Lemma~\ref{lemma:contraction_multiplier}.
\endproof

\begin{remark}
    \label{remark:novelty:implicit_lure}
    Our contractivity analysis includes matrices $C$ and $D$, whereas \cite{LDA-MC:13} considers $C=0$ and $D=I_{\ninput}$.
\end{remark}

\subsection{Feasibility of Sufficient Conditions for Contractivity}

\begin{proposition}[Feasibility of contractivity LMI]
\label{prop:contractivity_feasibility}
Consider the implicit Lur'e system \eqref{implicit_lure_system}.
Suppose $\Psi_t(z) = \alpha \Psi^\mathrm{base}_t(z)$ with $\alpha > 0$, and let $\Psi^\mathrm{base}_t: \R^{r} \to \mathcal{P}(\R^{\ninput})$ admit the IMMs
$
\begin{bmatrix}
    M_{11}^i & M_{12}^i
    \\
    (M_{12}^i)^\top & M_{22}^i
\end{bmatrix}
, i=1,\ldots,p
$
for each $t \in \timeset$, where $M_{11}^i \in \R^{r \times r}$, $M_{12}^i \in \R^{r \times \ninput}$, and $M_{22}^i \in \R^{\ninput \times \ninput}$.
Then $\Psi_t$ admits the IMMs
$
\begin{bmatrix}
    M_{11}^i & \frac{1}{\alpha} M_{12}^i \\
    \frac{1}{\alpha} (M_{12}^i)^\top & \frac{1}{\alpha^2} M_{22}^i
\end{bmatrix}
, i=1,\ldots,p
.
$
Furthermore, suppose that $D$ is invertible, $A-BD^{-1}C$ is Schur, and there exist nonnegative scalars $\bar{\lambda}_1,\dots,\bar{\lambda}_p$ such that
$
    \bar{M}_{22}
    \coloneqq
    \sum_{i=1}^p \bar{\lambda}_i M_{22}^i
    \prec 0.
$
Then, for any $\eta \in (\rho(A-BD^{-1}C), 1)$, there exist a positive definite matrix $P$, nonnegative scalars $\lambda_1,\dots,\lambda_p$, and a positive scalar $\alpha$ satisfying the contractivity inequality strictly, i.e.,
\begin{align}
    \label{contractivity_lmi:feasibility}
    \begin{multlined}
    \begin{bmatrix}
        A^\top P A - \eta^2 P & A^\top P B \\
        B^\top P A & B^\top P B
    \end{bmatrix}
    \\
    +
    \mathtt{sol}
    \Big(
    \sum_{i=1}^p \lambda_i
    \begin{bmatrix}
        M^i_{11} & \frac{1}{\alpha} M^i_{12} \\
        \frac{1}{\alpha} (M^i_{12})^\top & \frac{1}{\alpha^2} M^i_{22}
    \end{bmatrix}
    \Big)
    \prec 0
    .
    \end{multlined}
\end{align}
\end{proposition}

\smallskip

The proof is shown in Appendix~\ref{appendix:proof_implicit_lure}.

\subsection{Sufficient Conditions for Well-Posedness}
\label{sec:well_posedness}

We present two sufficient conditions for Assumption~\ref{assumption:generalized_equation}.
One follows from convex analysis and establishes well-posedness of the OCP for regularized MPC in Cases~1 and 2.

\begin{lemma}[Well-posedness via convex analysis]
\label{lemma:wellposedness_subdifferential}
Consider the generalized equation \eqref{generalized_equation} with $r = \ninput$, $G = I_{\ninput}$, and $D \succ 0$.
Suppose that, for each $t$, there exists a CCP function $f_t: \R^{\ninput} \to \R \cup \{+\infty\}$ such that $\Psi_t(z) = \partial f_t(z)$ for each $z$.
Then Assumption~\ref{assumption:generalized_equation} holds.
\end{lemma}

\proof
The generalized equation \eqref{generalized_equation} can be rewritten as $-Cx \in \partial_u \phi(u;x)$, where $\phi(u;x) = f(Fx + u) + \frac12 u^\top D u$.
Since $D \succ 0$, $\phi(u;x)$ is strongly convex with respect to $u$.
Therefore, the solution is unique for each $x$ and $t$.
\endproof

The other sufficient condition for Assumption~\ref{assumption:generalized_equation} is given by IMMs and a contractivity LMI, and it certifies the well-posedness of the OCP in Case~3 of regularized MPC.

\begin{proposition}[Well-posedness via contractivity LMI]
\label{prop:well-posedness_contraction}
Consider the implicit Lur'e system \eqref{implicit_lure_system}.
Suppose $D$ in \eqref{implicit_lure_plant} is invertible, and that $\Psi_t$ in \eqref{generalized_equation} admits $M_1, \dots, M_p \in \R^{(r+\ninput)\times(r+\ninput)}$ as IMMs.
Let $\eta \in (0,1)$.
If there exist a positive definite matrix $P \in \R^{\nstate \times \nstate}$ and nonnegative scalars $\lambda_1, \dots, \lambda_p$ such that \eqref{contractivity_lmi_implicit} holds strictly, $\bar{M}_{22} \prec 0$, and $\bar{M}_{12} = 0$, then Assumption~\ref{assumption:generalized_equation} holds.
\end{proposition}

The proof is shown in Appendix~\ref{appendix:proof_well-posedness}.

\section{Main Results: Contractivity Analysis of Regularized MPC}
\label{sec:main_results}

Using the methods developed in \S\ref{sec:implicit_lure}, we analyze the contractivity of the regularized MPC for the three cases defined in \S\ref{sec:problem_setup}.
The results are summarized in Table~\ref{table:summary}.

\begin{table*}[t]
\centering
\resizebox{.99\textwidth}{!}{
\begin{tabular}{|c|c|c|c|c|}
\hline
    Case &
    \begin{tabular}{c}
    Assumptions on the regularizer $\Vreg(X,U,t)$
    \end{tabular}
    &
    \begin{tabular}{c}
    Contractivity LMI
    \end{tabular}
    & 
    \begin{tabular}{c}
    Numerical studies
    \end{tabular}
\\
\hline
    1
    &
    \begin{tabular}{c}
    Sum of regularizing stage costs $\ellreg (u,t)$ \\
    which are convex and smooth
    \end{tabular}
    &
    Equation~\eqref{contractivity_lmi:convex_smooth} in Theorem~\ref{thm:contractivity_lmi:convex_smooth}
    &
    \begin{tabular}{c}
    Input soft penalty \\
    Fig.~\ref{fig:smooth} in \S\ref{sec:numerical_example_soft_input_penalty}
    \end{tabular}
\\
\hline
    2
    &
    \begin{tabular}{c}
    Sum of regularizing stage costs $\ellreg (u,t)$ \\
    which are convex, closed, and proper
    \end{tabular}
    &
    Equation~\eqref{contractivity_lmi:ccp} in Theorem~\ref{thm:contractivity_lmi:ccp}
    &
    \begin{tabular}{c}
    Input constraints
    \\
    Fig.~\ref{fig:tracking_convex} in \S\ref{sec:numerical_example_hard_input_penalty}
    \\
    \hline
    Sparse MPC
    \\
    Fig.~\ref{fig:consensus_sparse} in \S\ref{sec:numerical_example_sparse}
    \end{tabular}
\\
\hline
    3
    &
    The gradient of the regularizer is Lipschitz
    &
    Equation~\eqref{contractivity_lmi:lipschitz} in Theorem~\ref{thm:contractivity:lipschitz}
    &
    \begin{tabular}{c}
    State soft penalty \\
    Fig.~\ref{fig:full} in \S\ref{sec:numerical_example_soft_state_penalty}
    \end{tabular}
\\
\hline
\end{tabular}
}
\caption{Summary of our results.
\label{table:summary}}
\end{table*}

\subsection{Case 1: Convex smooth regularizers}

In this case, the gradient of the regularizing cost admits the IMM in Example~\ref{example:multiplier_convex_analysis}, which leads to the following theorem.

\begin{theorem}[Contractivity for convex smooth regularizers]
\label{thm:contractivity_lmi:convex_smooth}
Consider the regularized MPC under Assumptions~\ref{assumption:positive_definite_cost} and \ref{assumption:convex_smooth_regularizer}.
Let $\eta \in (0,1)$.
If there exist a positive definite matrix $P$ and nonnegative scalars $\lambda_1, \ldots, \lambda_H$ satisfying
\begin{align}
    \label{contractivity_lmi:convex_smooth}
    \begin{aligned}
    &
    \begin{bmatrix}
        A^\top P A - \eta^2 P & A^\top P B \Pi_1 \\
        \Pi_1^\top B^\top P A & \Pi_1^\top B^\top P B \Pi_1
    \end{bmatrix}
    \\
    &
    +
    \begin{bmatrix}
        0 & I_{H\ninput} \\ -C & -D
    \end{bmatrix}^\top
    \begin{bmatrix}
        0
        &
        \Lambda
        \\
        \Lambda
        &
        - \frac{2}{L} \Lambda
    \end{bmatrix}
    \begin{bmatrix}
        0 & I_{H\ninput} \\ -C & -D
    \end{bmatrix}
    \preceq 0
    ,
    \end{aligned}
\end{align}
where $\Lambda \coloneqq \diag{ \lambda_1, \dots, \lambda_H } \otimes I_{\ninput}$,
then the system is strongly contracting with factor $\eta$.
Furthermore, suppose $A_\mathrm{nom}$ in \eqref{A_nom} is Schur.
Then, for any $\eta \in (\rho(A_\mathrm{nom}), 1)$, there exist a positive definite matrix $P$, nonnegative scalars $\lambda_1, \dots, \lambda_H$, and a positive scalar $L$ satisfying \eqref{contractivity_lmi:convex_smooth}.
\end{theorem}

See Appendix \ref{appendix:proof_convex_smooth} for the proof of Theorem~\ref{thm:contractivity_lmi:convex_smooth}.
Since \eqref{contractivity_lmi:convex_smooth} is an LMI in $(P,\lambda_1,\dots,\lambda_H)$, its feasibility can be tested efficiently by semidefinite programming, and any feasible solution certifies contractivity.

\begin{example}
\label{example:design:input_soft_penalty}
Consider Example~\ref{example:input_soft_penalty}.
The regularizing cost $\ellreg$ in \eqref{ellreg:input_soft_penalty} is convex and $L$-smooth.
According to Theorem~\ref{thm:contractivity_lmi:convex_smooth}, \eqref{contractivity_lmi:convex_smooth} is sufficient for the contractivity of the regularized MPC and is feasible for a sufficiently small regularization scale $L$.
\end{example}

\subsection{Case 2: Convex closed proper regularizers}
\label{section:contractivity:ccp}

The subdifferential of a CCP function is monotone and admits the IMM in Example~\ref{example:multiplier_monotone}, yielding contractivity in Case~2.

\begin{theorem}[Contractivity for CCP regularizers]
\label{thm:contractivity_lmi:ccp}
Consider the regularized MPC under Assumptions~\ref{assumption:positive_definite_cost} and \ref{assumption:convex_closed_proper_regularizer}.
Let $\eta \in (0,1)$.
If there exist a positive definite matrix $P$ and nonnegative scalars $\lambda_1, \ldots, \lambda_H$ satisfying
\begin{align}
    \label{contractivity_lmi:ccp}
    \begin{aligned}
    &
    \begin{bmatrix}
        A^\top P A - \eta^2 P & A^\top P B \Pi_1 \\
        \Pi_1^\top B^\top P A & \Pi_1^\top B^\top P B \Pi_1
    \end{bmatrix}
    \\
    &
    +
    \begin{bmatrix}
        0 & I_{H\ninput} \\ -C & -D
    \end{bmatrix}^\top
    \begin{bmatrix}
        0
        &
        \Lambda
        \\
        \Lambda
        &
        0
    \end{bmatrix}
    \begin{bmatrix}
        0 & I_{H\ninput} \\ -C & -D
    \end{bmatrix}
    \preceq 0
    ,
    \end{aligned}
\end{align}
where $\Lambda \coloneqq \diag{ \lambda_1, \dots, \lambda_H } \otimes I_{\ninput}$,
then the system is strongly contracting with factor $\eta$.
\end{theorem}

See Appendix \ref{appendix:proof_ccp} for the proof of Theorem~\ref{thm:contractivity_lmi:ccp}.
As in Examples \ref{example:input_hard_constraint} and \ref{example:sparse}, the CCP regularization includes hard constraints of inputs and sparse control.
Since \eqref{contractivity_lmi:ccp} is an LMI in $(P,\lambda_1,\dots,\lambda_H)$, its feasibility can be tested efficiently by semidefinite programming, and any feasible solution certifies contractivity.

\begin{remark}
\label{remark:open_loop_stability}
We note that \eqref{contractivity_lmi:ccp} requires the open-loop stability of the plant, i.e., $A^\top P A - \eta^2 P \preceq 0$,
since Theorem~\ref{thm:contractivity_lmi:ccp} certifies global contractivity even in the case of $u=0$, i.e., $\mathcal{U}_t = \{0\}$ in Example~\ref{example:input_hard_constraint}.
This limitation could be relaxed if we only require local contractivity, which is a topic of future work.
Furthermore, if the regularizer is also $L$-smooth, the problem falls under Case~1, in which closed-loop contractivity can be guaranteed even when the open-loop system is unstable.
\end{remark}

\begin{remark}
\label{remark:novelty:ccp}
As used in \cite[Proposition 2]{SH-GB-FD-DLMP:26}, convex analysis and monotone operator theory show that the inverse mapping of a $\mu$-strongly monotone operator is $\mu$-cocoercive \cite[Ex. 22.7]{HHB-PLC:17}.
Our analysis captures richer information in two respects:
i) Theorem~\ref{thm:contractivity_lmi:ccp} allows coefficients $\lambda$s, since the regularizer is assumed to be a sum of stage-wise regularizing costs; and
ii) the matrix $D$ in the generalized equation is explicitly exploited, whereas $\mu$-monotonicity characterizes $D$ solely through its maximum eigenvalue.
See Appendix \ref{appendix:comparison_multipliers} for a detailed comparison.
\end{remark}

\subsection{Case 3: Lipschitz-gradient regularizers}

In Case~3, the gradient of the regularizer admits the IMM in Example~\ref{example:multiplier_lipschitz}, yielding the following theorem.

\begin{theorem}[Contractivity for Lipschitz-gradient regularizers]
\label{thm:contractivity:lipschitz}
Consider the regularized MPC under Assumptions~\ref{assumption:lipschitz:well-posedness} and \ref{assumption:gradient_lipschitz_regularizer}.
Let $\eta \in (0,1)$ and $\gamma = L^{-2}$, where $L$ is the Lipschitz constant of the gradient of the regularizer.
Suppose there exists a positive definite matrix $P$ such that
\begin{align}
    \label{contractivity_lmi:lipschitz}
    \begin{aligned}
    &
    \begin{bmatrix}
        A^\top P A - \eta^2 P & A^\top P B \Pi_1 \\
        \Pi_1^\top B^\top P A & \Pi_1^\top B^\top P B \Pi_1
    \end{bmatrix}
    \\
    &
    +
    \begin{bmatrix}
        \bar{A} & \bar{B}
        \\
        0 & I_{H\ninput}
        \\
        -C & -D
    \end{bmatrix}^\top
    \begin{bmatrix}
    I_{H\nstate+H\ninput} & 0
    \\
    0 & - \gamma I_{H\ninput}
    \end{bmatrix}
    \begin{bmatrix}
        \bar{A} & \bar{B}
        \\
        0 & I_{H\ninput}
        \\
        -C & -D
    \end{bmatrix}
    \preceq 0
    .
    \end{aligned}
\end{align}
Then the system is strongly contracting with factor $\eta$.
\end{theorem}

See Appendix \ref{appendix:proof_lipschitz} for the proof of Theorem~\ref{thm:contractivity:lipschitz}.
The matrix inequality \eqref{contractivity_lmi:lipschitz} is an LMI in $P$ and $\gamma = L^{-2}$.
Therefore, by solving a semidefinite program that minimizes $\gamma$ subject to the LMI constraints, one can search for the largest $L$ for which the closed-loop system is contracting.
Furthermore, the feasibility of the LMI and the well-posedness of the regularized MPC are guaranteed by the following theorem.

\begin{theorem}[Feasibility of LMI and well-posedness of OCP]
\label{thm:lipschitz:feasibility_wellposedness}
Consider the regularized MPC under Assumptions~\ref{assumption:positive_definite_cost} and \ref{assumption:gradient_lipschitz_regularizer}.
Suppose $A_\mathrm{nom}$ in \eqref{A_nom} is Schur.
Then, for any $\eta \in (\rho(A_\mathrm{nom}), 1)$, there exist a positive definite matrix $P$ and a positive scalar $\gamma$ such that the strict version of \eqref{contractivity_lmi:lipschitz} holds.
Furthermore, Assumption~\ref{assumption:lipschitz:well-posedness} holds if the regularizer is scaled so that the Lipschitz constant of its gradient is $L=\gamma^{-\frac12}$.
\end{theorem}

See Appendix \ref{appendix:proof_lipschitz:feasibility_wellposedness} for the proof of Theorem~\ref{thm:lipschitz:feasibility_wellposedness}.

\begin{example}[Soft penalty for state constraints]
\label{example:design:state_soft_penalty}
Consider Example~\ref{example:state_soft_penalty}.
The regularizer in \eqref{Vreg:state_soft_penalty} is Lipschitz with respect to $(X,U)$ with Lipschitz constant at most $L$.
Theorem~\ref{thm:contractivity:lipschitz} enables an SDP-based search for the largest regularization scale $L$ for which the closed-loop system of the regularized MPC remains strongly contracting with a prescribed factor $\eta$.
Furthermore, Assumption~\ref{assumption:lipschitz:well-posedness} holds if the regularizer is scaled so that the Lipschitz constant of its gradient equals the obtained value $L$.
\end{example}

\section{Numerical Case Studies}
\label{sec:numerical_examples}

This section illustrates the contractivity conditions developed above using four MPC designs.
The first example treats a convex smooth input soft penalty.
The next two examples use CCP regularizers for hard input constraints and sparse control, respectively.
The final example treats a state soft penalty with a Lipschitz gradient.
In each case, we solve the corresponding LMI as a semidefinite program, report a feasible contraction metric, and compare the closed-loop behavior of the nominal and regularized MPC laws.
The simulation code is available at \url{https://github.com/ToyotaCRDL/implicit-lure-mpc.git}.

\begin{figure}[t]
\centering
\begin{tikzpicture}[line cap=round,line join=round,>=Stealth]
  \def\gapY{1.0}    
  \def\spanX{1.6}   
  \def\massW{1.8}   
  \def\massH{2.2}   

  \def\damperBody{0.60}
  \def\damperPlate{0.35}
  \def\damperHalfH{0.25}

  \def\springLen{1.2}

  \def\labelUp{6pt}

  \tikzset{
    mass/.style   = {draw, thick, fill=gray!10, minimum width=\massW cm, minimum height=\massH cm},
    kspring/.style= {thick, decorate, decoration={zigzag, segment length=6pt, amplitude=2.5pt}},
    thinlabel/.style={font=\small, inner sep=1pt},
  }

  \pgfmathsetmacro\yU{\gapY/2}
  \pgfmathsetmacro\yL{-\gapY/2}

  \node[mass, anchor=west] (m1) at (\spanX,0) {$m_1$};
  \node[mass, anchor=west] (m2) at ({2*\spanX+\massW},0) {$m_2$};

  \coordinate (m1Wtop) at ($(m1.west)+(0,\yU)$);
  \coordinate (m1Wbot) at ($(m1.west)+(0,\yL)$);
  \coordinate (m1Etop) at ($(m1.east)+(0,\yU)$);
  \coordinate (m1Ebot) at ($(m1.east)+(0,\yL)$);
  \coordinate (m2Wtop) at ($(m2.west)+(0,\yU)$);
  \coordinate (m2Wbot) at ($(m2.west)+(0,\yL)$);

  \draw[thick] (-0.6, 1.8) -- (-0.6,-1.9);
  \foreach \y in {-1.8,-1.5,...,1.7} {
    \draw[thick] (-0.6,\y) -- (-0.9,\y+0.15);
  }
  \coordinate (wallTop) at (0,\yU);
  \coordinate (wallBot) at (0,\yL);
  \draw[thick] (-0.6,\yU) -- (wallTop);
  \draw[thick] (-0.6,\yL) -- (wallBot);

  \newcommand{\CenteredSpring}[4]{%
    \coordinate (mid) at ($(#1)!0.5!(#2)$);
    \draw[thick] (#1) -- ($(mid)+(-\springLen/2,0)$);
    \draw[kspring] ($(mid)+(-\springLen/2,0)$) -- ($(mid)+(\springLen/2,0)$);
    \draw[thick] ($(mid)+(\springLen/2,0)$) -- (#2);
    \node[thinlabel, #4=\labelUp] at (mid) {#3};
  }

  \CenteredSpring{wallTop}{m1Wtop}{$k_1$}{above}
  \CenteredSpring{m1Etop}{m2Wtop}{$k_2$}{above}

  \coordinate (midc1) at ($(wallBot)!0.5!(m1Wbot)$);
  \draw[thick] ($(midc1)+(-\damperBody/2,-\damperHalfH)$) rectangle
               ($(midc1)+(\damperBody/2,\damperHalfH)$);
  \draw[thick] ($(midc1)+(-\damperBody/2,0)$) -- ++(-\damperPlate,0);
  \draw[thick] ($(midc1)+(\damperBody/2,0)$) -- ++(\damperPlate,0);
  \draw[thick] (wallBot) -- ($(midc1)+(-\damperBody/2-\damperPlate,0)$);
  \draw[thick] ($(midc1)+(\damperBody/2+\damperPlate,0)$) -- (m1Wbot);
  \node[thinlabel, above=\labelUp, yshift=1mm] at (midc1) {$c_1$};

  \coordinate (midc2) at ($(m1Ebot)!0.5!(m2Wbot)$);
  \draw[thick] ($(midc2)+(-\damperBody/2,-\damperHalfH)$) rectangle
               ($(midc2)+(\damperBody/2,\damperHalfH)$);
  \draw[thick] ($(midc2)+(-\damperBody/2,0)$) -- ++(-\damperPlate,0);
  \draw[thick] ($(midc2)+(\damperBody/2,0)$) -- ++(\damperPlate,0);
  \draw[thick] (m1Ebot) -- ($(midc2)+(-\damperBody/2-\damperPlate,0)$);
  \draw[thick] ($(midc2)+(\damperBody/2+\damperPlate,0)$) -- (m2Wbot);
  \node[thinlabel, above=\labelUp, yshift=1mm] at (midc2) {$c_2$};

  \draw[->, thick] (m1.north) -- ++(1.5,0) node[midway, above] {$u_1$};
  \draw[->, thick] (m2.north) -- ++(1.5,0) node[midway, above] {$u_2$};
\end{tikzpicture}
\caption{Series mass-spring-damper system used in the mass-spring-damper case studies.
The state is $x=[q_1;\,v_1;\,q_2;\,v_2]$, and the control inputs $u_1$ and $u_2$ are external forces applied to the two masses.}
\label{fig:msd}
\end{figure}
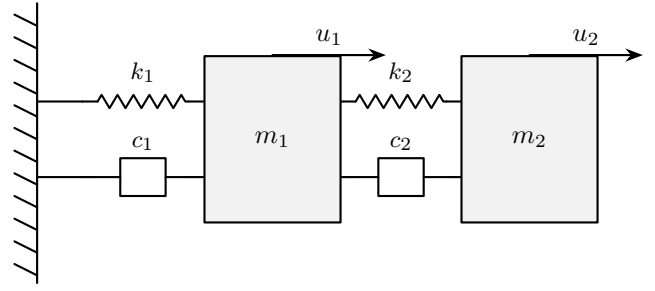

\subsection{MPC of Series Mass-Spring-Damper System with Soft Penalty on Inputs}
\label{sec:numerical_example_soft_input_penalty}

We first illustrate Theorem~\ref{thm:contractivity_lmi:convex_smooth} on an MPC design with a soft penalty on input constraints.
Consider the series mass-spring-damper system shown in Figure~\ref{fig:msd}.
Let $x = [q_1; v_1; q_2; v_2]$ be the state vector, where $q_i$ and $v_i$ are the position and velocity of mass $i$, respectively, and let $u = [u_1; u_2]$ be the vector of applied forces.
The continuous-time dynamics are
\begin{align*}
    \dot{x}(t) = A_\mathrm{c} x(t) + B_\mathrm{c} u(t)
    ,
    \\
    A_\mathrm{c}
    =
    \begin{bmatrix}
        0 & 1 & 0 & 0
        \\
        a_{21} & a_{22} & a_{23} & a_{24}
        \\
        0 & 0 & 0 & 1
        \\
        \frac{k_2}{m_2} & \frac{c_2}{m_2} & -\frac{k_2}{m_2} & -\frac{c_2}{m_2}
    \end{bmatrix}
    ,
    \ \
    B_\mathrm{c}
    =
    \begin{bmatrix}
        0 & 0
        \\
        \frac{1}{m_1} & 0
        \\
        0 & 0
        \\
        0 & \frac{1}{m_2}
    \end{bmatrix}
\end{align*}
where $a_{21} = -\frac{k_1 + k_2}{m_1}$, $a_{22} = -\frac{c_1 + c_2}{m_1}$, $a_{23} = \frac{k_2}{m_1}$, and $a_{24} = \frac{c_2}{m_1}$.
The parameters are $m_1=9, m_2=10, k_1=5, k_2=6$, and $c_1=c_2=0$.
Applying a zero-order hold with sampling period $\varDelta T = 1$ gives the discrete-time model \eqref{plant_lti}.

For the nominal MPC, we set $Q = 5 I_4$ and $R=I_2$, and choose the terminal weight $Q_\mathrm{f}$ as the solution of DARE \eqref{dare}.
The prediction horizon is $H=5$.
In this setting, $\rho(A_\mathrm{nom}) = 0.867$.

To encourage the inputs to satisfy $\underline{u} \leq u_i \leq \bar{u}$ $(i=1,2)$ with $\underline{u} = -1.0$ and $\bar{u} = 0.5$, we use the soft penalty in Example~\ref{example:input_soft_penalty} with $L=10$.
We choose the desired contractivity factor $\eta = 0.99$, which is larger than $\rho(A_\mathrm{nom})$.
The LMI \eqref{contractivity_lmi:convex_smooth} is feasible; therefore, by Theorem~\ref{thm:contractivity_lmi:convex_smooth}, the closed-loop system is strongly contracting with factor $\eta$ with respect to the norm $\|\cdot\|_P$, where $P$ is the LMI solution.
Because the regularized MPC satisfies $u=0$ at $x=0$, Lemma~\ref{lemma:contraction_benefit} further guarantees exponential stability of the origin.

To illustrate the certified contraction, we simulate two trajectories $\{x^{(1)}(t)\}$ and $\{x^{(2)}(t)\}$ starting from
$
    x^{(1)}(0) = [-1;\, -0.8;\, 1;\, 0.5], \ \ 
    x^{(2)}(0) = [1;\, -1;\, 1;\, 1]
    .
$
Figure~\ref{fig:smooth} shows the resulting trajectories under the nominal and regularized MPC laws.
The top panel shows $\|x^{(1)}(t) - x^{(2)}(t)\|_P$ on a logarithmic scale.
For comparison, the exponential decay with rate $\eta$ is shown in gray.
Both the nominal and regularized MPC laws exhibit contraction with factor $\eta$, whereas the open-loop system does not.

The middle and bottom panels show the trajectories of $u_1$ and $u_2$, respectively.
The regularized MPC keeps the inputs closer to the interval $[-1,\,0.5]$, as intended by the input penalty.
To quantify this effect, we measure the cumulative constraint violation
$
\sum_{t=0}^{39} \sum_{i=1}^2 \max\{0, u_i(t) - \bar{u}, \underline{u} - u_i(t)\}
$.
For 50 initial states sampled uniformly from $[-3.0, 3.0]^4$, the regularized MPC gives a smaller cumulative violation than the nominal MPC for every initial condition.
The average cumulative violation is $18.3$ for the regularized MPC and $36.7$ for the nominal MPC.

\begin{figure}[t]
\centering
\includegraphics[width=0.99\linewidth]{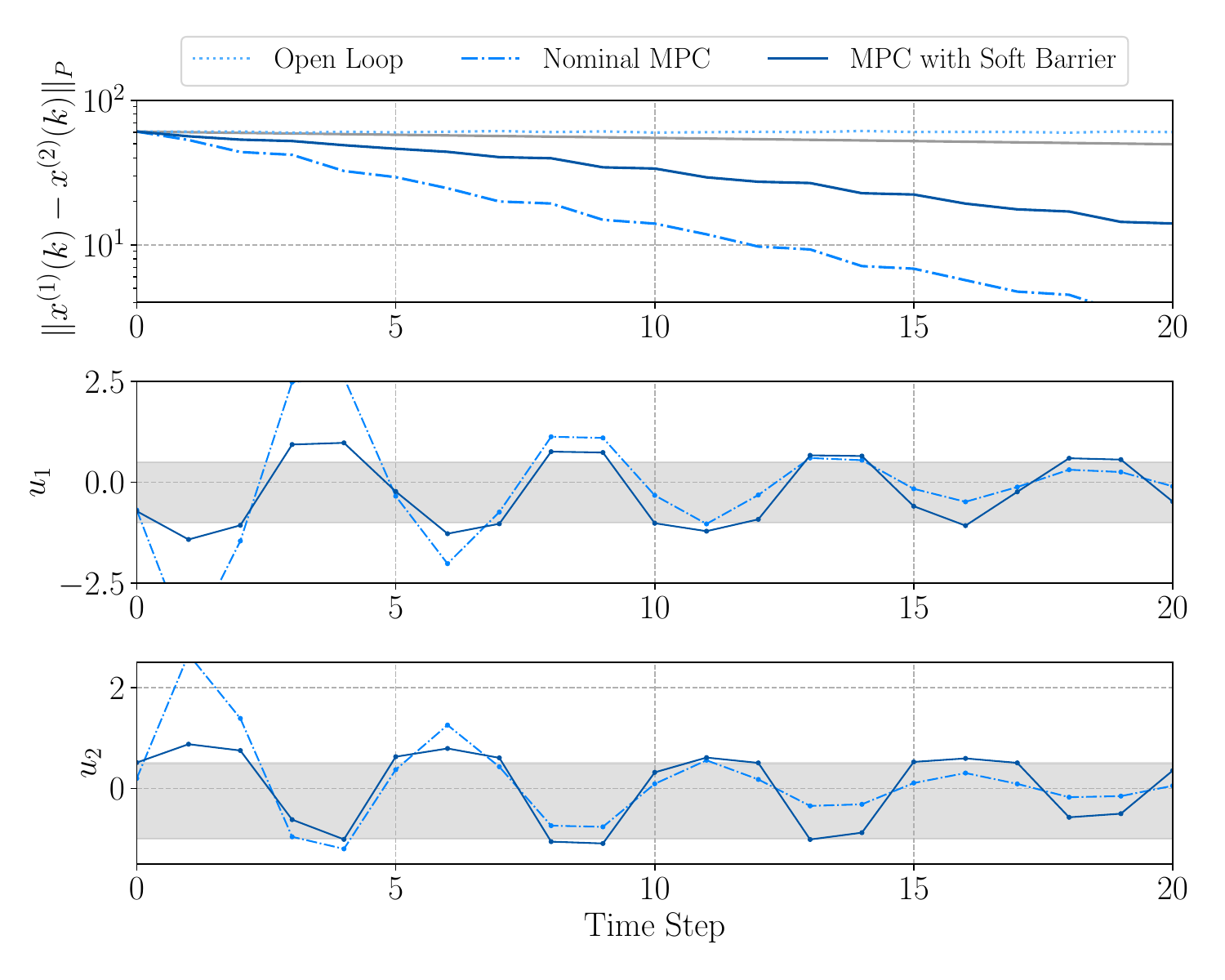}
\caption{Soft input-penalty example.
Top: distance between two closed-loop trajectories, measured in the certified norm $\| \cdot \|_P$, with the gray line indicating the prescribed decay rate $\eta$.
Middle and bottom: input trajectories $u_1$ and $u_2$.
The input-penalized MPC preserves the certified contraction behavior while keeping the inputs closer to the desired interval $[-1,\,0.5]$ than the nominal MPC.}
\label{fig:smooth}
\end{figure}

\subsection{Tracking MPC for Series Mass-Spring-Damper System with Hard Constraints on Inputs}
\label{sec:numerical_example_hard_input_penalty}

We next illustrate Theorem~\ref{thm:contractivity_lmi:ccp} on a tracking MPC design with hard input constraints.

We consider the series mass-spring-damper system shown in Figure~\ref{fig:msd} with parameters $m_1 = 9, m_2 = 8, k_1 = 7, k_2 = 6, c_1 = 5, c_2 = 4$ and sampling period $\varDelta T = 1$.
We choose the reference positions
$x_1^\mathrm{ref}(t) = 2\sin{(\frac{\pi}{4}t)}$ and
$x_3^\mathrm{ref}(t) = 4\sin{(\frac{\pi}{4}t)}$, and define
$q^\mathrm{ref}(t) \coloneqq [x_1^\mathrm{ref}(t); x_3^\mathrm{ref}(t)]$.
We assume that the reference state $x^\mathrm{ref}(t)$ and reference input $u^\mathrm{ref}(t)$ satisfy the state-space model $x^+ = Ax + Bu$ and are periodic with period 16, i.e., $x^\mathrm{ref}(t+16) = x^\mathrm{ref}(t), \ u^\mathrm{ref}(t+16) = u^\mathrm{ref}(t)$ for all $t \in \timeset$.
Then $x^\mathrm{ref}(1), \ldots, x^\mathrm{ref}(16)$ and $u^\mathrm{ref}(1), \ldots, u^\mathrm{ref}(16)$ satisfy
\begin{align*}
A x^\mathrm{ref}(t) - x^\mathrm{ref}(t+1) + B u^\mathrm{ref}(t) &= 0
\ \
t = 1, \dots, 15
,
\\
A x^\mathrm{ref}(16) - x^\mathrm{ref}(1) + B u^\mathrm{ref}(16) &= 0
,
\\
\Pi_\mathrm{pos} x^\mathrm{ref}(t) - q^\mathrm{ref}(t) &= 0
\ \
t = 1, \dots, 16
,
\\
\Pi_\mathrm{pos}
&\coloneqq
\begin{bmatrix}
    1 & 0 & 0 & 0 \\ 0 & 0 & 1 & 0
\end{bmatrix}
.
\end{align*}

The tracking error is defined as $x^\mathrm{e} \coloneqq x - x^\mathrm{ref}(t)$ and the feedback input as $u^\mathrm{e} \coloneqq u - u^\mathrm{ref}(t)$.
By linearity of \eqref{plant_lti}, the error variables satisfy $x^\mathrm{e}(t+1) = Ax^\mathrm{e}(t) + Bu^\mathrm{e}(t)$.
The MPC control law for the error system is given by $u^\mathrm{e}(t) = \Pi_1 U^\ast(x^\mathrm{e}(t), t)$ as in \eqref{mpc_control_law}, yielding $u(t) = u^\mathrm{ref}(t) + \Pi_1 U^\ast(x^\mathrm{e}(t), t)$.
We set $Q = 10 I_4, R = I_2$ and select the terminal weight $Q_\mathrm{f}$ as the solution to DARE in \eqref{dare}.
The prediction horizon is $H = 15$.
In this setting, $\rho(A_\mathrm{nom}) = 0.693$.

We impose the input box constraints $-6 \leq u_i \leq 6$ $(i = 1, 2)$, which become the time-varying constraints
$-6 - u^\mathrm{ref}_i(t) \leq u_i^\mathrm{e} \leq 6 - u^\mathrm{ref}_i(t)$ on the feedback input.
These constraints are represented by selecting the regularizer as the indicator function in Example~\ref{example:input_hard_constraint}.
We choose the desired contraction factor $\eta = 0.95$.
The LMI \eqref{contractivity_lmi:ccp} is feasible; therefore, by Theorem~\ref{thm:contractivity_lmi:ccp}, the regularized MPC closed loop for the error system is strongly contracting with factor $\eta$ with respect to the norm $\|\cdot\|_P$, where $P$ is the LMI solution.
Consequently, the tracking MPC is contracting.
Since the control law is periodic, Lemma~\ref{lemma:contraction_benefit} implies that the closed-loop state and input trajectories converge to a periodic orbit.

To illustrate the certified contraction, we simulate two trajectories $\{x^{(1)}(k)\}$ and $\{x^{(2)}(k)\}$ starting from
$
    x^{(1)}(0) = [4;\, 0;\, 4;\, 0], \ \ 
    x^{(2)}(0) = [-4;\, 0;\, -4;\, 0]
    .
$
Figure~\ref{fig:tracking_convex} shows the results for the nominal and hard-constrained MPC laws.
The top panel shows $\|x^{(1)}(t) - x^{(2)}(t)\|_P$ on a logarithmic scale.
For comparison, the exponential decay with rate $\eta$ is shown in gray.
The open-loop system, nominal MPC, and hard-constrained MPC all exhibit contraction with factor $\eta$ in this example.
The middle and bottom panels show the trajectories of $u_1$ and $u_2$, respectively.
The nominal MPC achieves accurate tracking but violates the input constraints, whereas the hard-constrained MPC respects the constraints and tracks the reference as closely as the input bounds allow.
The constrained closed-loop input also converges to a periodic orbit.

\begin{figure}[t]
\centering
\includegraphics[width=0.99\linewidth]{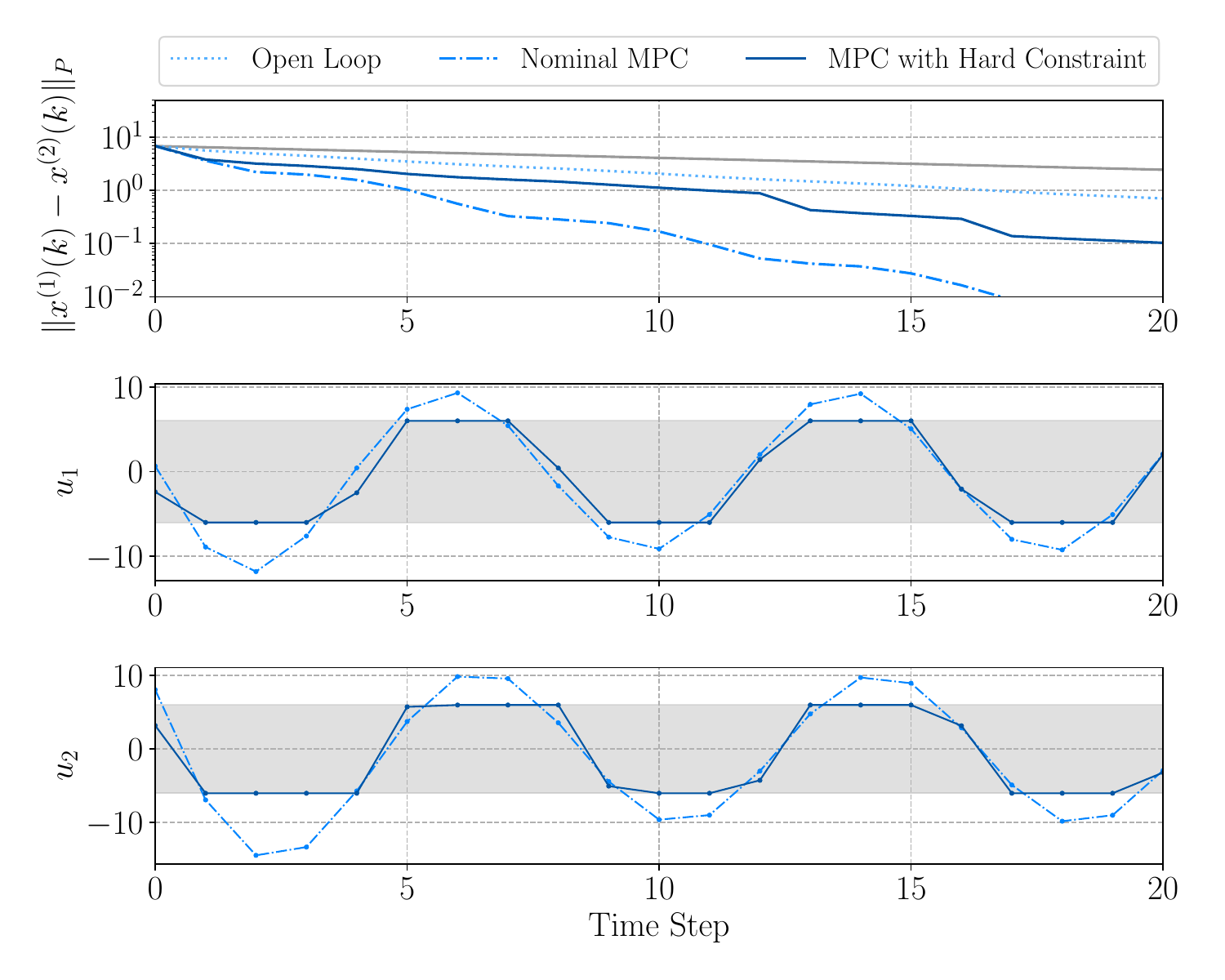}
\caption{
Tracking example with hard input constraints.
Top: distance between two error trajectories, measured in the certified norm $\| \cdot \|_P$, with the gray line indicating the prescribed decay rate $\eta$.
Middle and bottom: input trajectories $u_1$ and $u_2$.
The nominal MPC tracks accurately but can violate the input bounds, whereas the hard-constrained MPC enforces $-6 \leq u_i \leq 6$ and converges to the constrained periodic response.
\label{fig:tracking_convex}
}
\end{figure}

\subsection{Sparse Control of Consensus Networks}
\label{sec:numerical_example_sparse}

We now consider a consensus network represented by a graph $G = (V,E)$, where each node $i \in V$ has a state $z_i(t) \in \R$ and edge inputs are used to promote consensus.
The dynamics of the consensus network is described by
\begin{align}
    \label{consensus}
    z(t+1) = (I - \varepsilon \alpha L_G) z(t) + \varepsilon B_G u(t),
    \\
    \notag
    L_G =
    \begin{bmatrix}
         2 & -1 &  0 &  0 & -1 &  0 \\
        -1 &  3 & -1 &  0 & -1 &  0 \\
         0 & -1 &  3 & -1 &  0 & -1 \\
         0 &  0 & -1 &  2 &  0 & -1 \\
        -1 & -1 &  0 &  0 &  3 & -1 \\
         0 &  0 & -1 & -1 & -1 &  3
    \end{bmatrix}
    ,
    \\
    \notag
    B_G =
    \begin{bmatrix}
         0 &  0 &  0 &  0 &  0 &  0  &  1 &  1 \\
         1 &  0 &  1 &  0 &  0 &  0  & -1 &  0 \\
        -1 &  1 &  0 &  1 &  0 &  0  &  0 &  0 \\
         0 & -1 &  0 &  0 &  0 &  1  &  0 &  0 \\
         0 &  0 & -1 &  0 &  1 &  0  &  0 & -1 \\
         0 &  0 &  0 & -1 & -1 & -1  &  0 &  0
    \end{bmatrix}
\end{align}
with $\varepsilon = 0.2$ and $\alpha = 0.1$.
Here, $z(t) \in \R^6$ is the state vector, $u(t) \in \R^8$ is the edge-input vector, and $L_G$ and $B_G$ are the graph Laplacian and incidence matrix, respectively.
The parameter $\varepsilon$ is the sampling period, and $\alpha$ is the consensus gain.
The communication graph is shown in Figure~\ref{fig:graph}.
One instance of such dynamics is a vehicle platoon, where each node represents a vehicle, $z_i(t)$ denotes its position, $(I - \varepsilon \alpha L_G)z(t)$ describes local feedback from neighboring vehicles, and $\varepsilon B_G u(t)$ represents external velocity commands from a central coordinator.

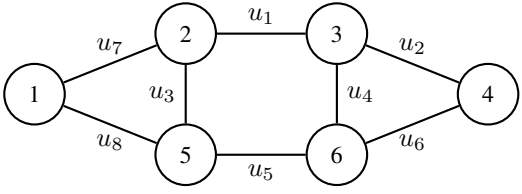
\begin{figure}[t]
\centering
\begin{tikzpicture}
[
>=stealth,
node/.style={circle, draw, thick, minimum size=8mm, inner sep=0pt, font=\small},
edge/.style={thick}
]
\node[node] (n1) at (-2,-0.8) {1};
\node[node] (n2) at (0,0)    {2};
\node[node] (n3) at (2,0)    {3};
\node[node] (n4) at (4,-0.8) {4};
\node[node] (n5) at (0,-1.6) {5};
\node[node] (n6) at (2,-1.6) {6};
\draw[edge] (n1) -- node[above] {$u_7$} (n2);
\draw[edge] (n1) -- node[below] {$u_8$} (n5);
\draw[edge] (n2) -- node[above] {$u_1$} (n3);
\draw[edge] (n3) -- node[above] {$u_2$} (n4);
\draw[edge] (n2) -- node[left]  {$u_3$} (n5);
\draw[edge] (n3) -- node[right] {$u_4$} (n6);
\draw[edge] (n5) -- node[below] {$u_5$} (n6);
\draw[edge] (n4) -- node[below] {$u_6$} (n6);
\end{tikzpicture}
\caption{Communication graph of the consensus network used in the sparse-control example.
Nodes represent scalar agents, and edge labels indicate the eight control inputs acting through the incidence matrix $B_G$.}
\label{fig:graph}
\end{figure}

For any input $u(t)$, the state-space model \eqref{consensus} preserves the average $\bar{z} \coloneqq \frac16 \sum_{i=1}^6 z_i(t)$ because $\mathbf{1}^\top L_G = 0$ and $\mathbf{1}^\top B_G = 0$, where $\mathbf{1} \coloneqq [1; \dots; 1] \in \R^6$.
To analyze convergence to the average, define the deviation variables $x_i(t) \coloneqq z_i(t) - \bar z$ for $i=1,\dots,6$.
Since $\sum_{i=1}^6 x_i(t)=0$, we eliminate $x_6$ with $x_6 = -\sum_{i=1}^5 x_i$ and define the reduced state as $x(t) \coloneqq [x_1;\dots;x_5]$.
The dynamics of the reduced state is given by
\begin{align}
    \label{consensus_reduced}
    x(t+1) = A x(t) + B u(t)
\end{align}
where $A$ and $B$ are given by
\begin{align*}
    A &\coloneqq
    \begin{bmatrix}
        1 - 2 \varepsilon \alpha &  \varepsilon \alpha        & 0            & 0            &  \varepsilon \alpha \\
         \varepsilon \alpha      & 1 - 3 \varepsilon \alpha &  \varepsilon \alpha      & 0            &  \varepsilon \alpha \\
        - \varepsilon \alpha     & 0              & 1 - 4 \varepsilon \alpha & 0            & - \varepsilon \alpha \\
        - \varepsilon \alpha     & - \varepsilon \alpha      & 0             & 1 - 3 \varepsilon \alpha & - \varepsilon \alpha \\
        0             & 0              & - \varepsilon \alpha     & - \varepsilon \alpha     & 1 - 4 \varepsilon \alpha
    \end{bmatrix}
    ,
    \\
    B &\coloneqq \varepsilon
    \begin{bmatrix}
         0 &  0 &  0 & 0 & 0 & 0 &  1 &  1 \\
         1 &  0 &  1 & 0 & 0 & 0 & -1 &  0 \\
        -1 &  1 &  0 & 1 & 0 & 0 &  0 &  0 \\
         0 & -1 &  0 & 0 & 0 & 1 &  0 &  0 \\
         0 &  0 & -1 & 0 & 1 & 0 &  0 & -1
    \end{bmatrix}
    .
\end{align*}

To promote consensus using sparse control effort, we use the following stage cost:
\begin{align}
    \label{stage_consensus_z}
    \hat{\ell}(z,u) \coloneqq \sum_{i=1}^6 \frac12 \left\| z_i - \bar{z} \right\|^2 + \|u\|^2 + \lambda \|u\|_1
    .
\end{align}
In terms of the reduced state $x$, this stage cost becomes
\begin{align*}
    \ell(x,u) \coloneqq \frac12 x^\top Q x + \frac12 u^\top R u + \lambda \|u\|_1
    ,
    \\ 
    Q \coloneqq I_5 + \mathbf{1}_5\mathbf{1}_5^\top
    ,
    \ \ 
    R \coloneqq 2 I_8
    .
\end{align*}
Here, $\mathbf{1}_5 \in \R^5$ denotes the all-ones vector.
We set $H=10$ and choose the terminal weight $Q_\mathrm{f}$ as the solution of DARE in \eqref{dare}.
We choose the desired contractivity factor $\eta = 0.99$, which is larger than $\rho(A_\mathrm{nom}) = 0.866$.
Then the LMI \eqref{contractivity_lmi:ccp} is feasible; therefore, by Theorem~\ref{thm:contractivity_lmi:ccp}, the closed-loop system of the regularized MPC is strongly contracting with factor $\eta$ with respect to the norm $\|\cdot\|_P$, where $P$ is the LMI solution.
Since $u=0$ at $x=0$, Lemma~\ref{lemma:contraction_benefit} implies exponential stability of the consensus subspace.

To illustrate the certified contraction, we compute two trajectories $\{x^{(1)}(t)\}$ and $\{x^{(2)}(t)\}$ starting from
$
    x^{(1)}(0) = [-3;\, 3;\, 2;\, 1;\, -5], \ \ 
    x^{(2)}(0) = [ 1;\, -1;\, 1;\, 1;\, 3]
    .
$
Figure~\ref{fig:consensus_sparse} shows the results for the nominal and sparse MPC laws.
The top panel shows $\|x^{(1)}(t) - x^{(2)}(t)\|_P$ on a logarithmic scale.
For comparison, the exponential decay with rate $\eta$ is shown in gray.
The open-loop system, nominal MPC, and sparse MPC all exhibit contraction with factor $\eta$ in this example.
The middle and bottom panels show the trajectories of $u_2$ and $u_4$, respectively.
The sparse MPC turns off $u_2$ after time step 16 and $u_4$ after time step 5, demonstrating the sparsity-inducing effect of the $\ell_1$ regularizer.

Figure~\ref{fig:consensus_stat} plots the number of zero-valued inputs for 50 randomly generated initial conditions and several values of the regularization coefficient $\lambda$ in \eqref{stage_consensus_z}.
As $\lambda$ increases, the controller uses fewer nonzero inputs.

\begin{figure}[t]
  \centering
  \includegraphics[width=0.99\linewidth]{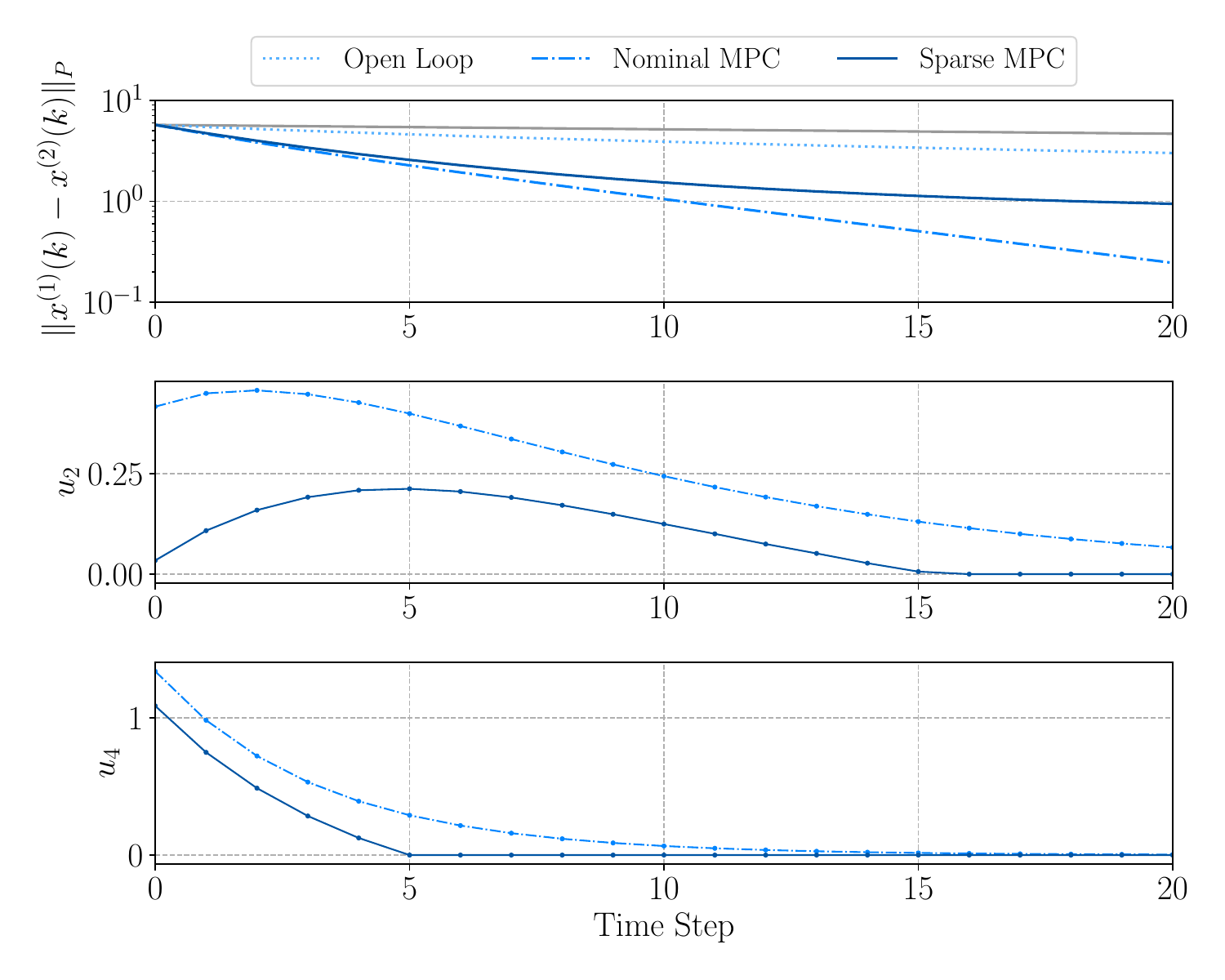}
  \caption{Sparse consensus-control example with $\lambda=1.5$.
  Top: distance between two reduced-state trajectories, measured in the certified norm $\| \cdot \|_P$, with the gray line indicating the prescribed decay rate $\eta$.
  Middle and bottom: representative edge inputs $u_2$ and $u_4$.
  The sparse MPC retains contraction while driving selected inputs exactly to zero, illustrating the sparsity-promoting effect of the $\ell_1$ regularizer.
  \label{fig:consensus_sparse}
  }
\end{figure}

\begin{figure}[t]
  \centering
  \includegraphics[width=0.99\linewidth]{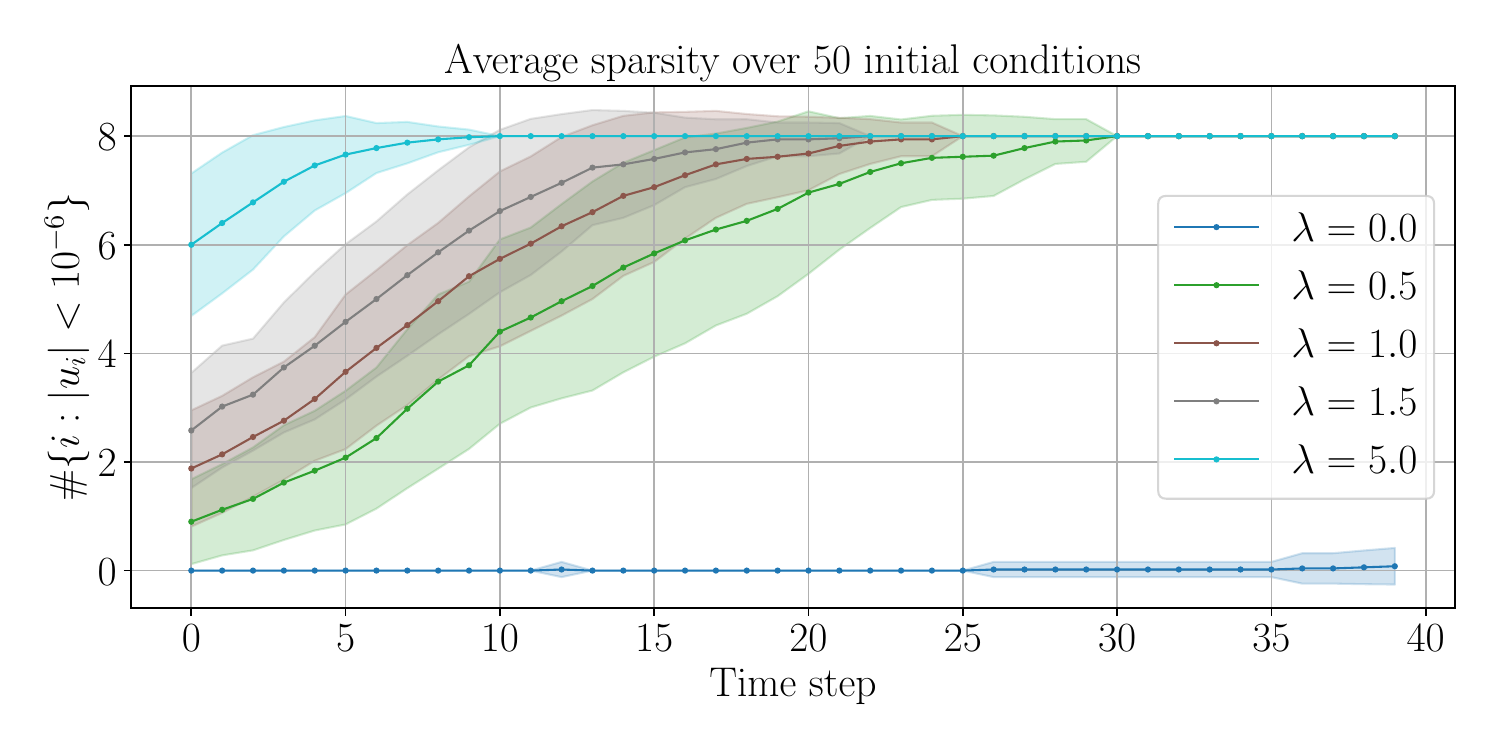}
  \caption{Sparsity statistics for the consensus example over 50 randomly generated initial conditions.
  The curves show the number of zero-valued edge inputs over time for different values of the regularization coefficient $\lambda$.
  Larger $\lambda$ values produce sparser control actions, whereas the nominal MPC case ($\lambda=0$) rarely sets inputs exactly to zero.
  \label{fig:consensus_stat}
  }
\end{figure}

\subsection{MPC of Series Mass-Spring-Damper System with Soft Penalty on States}
\label{sec:numerical_example_soft_state_penalty}

Finally, we illustrate Theorem~\ref{thm:contractivity:lipschitz} on an MPC design with a soft penalty on the states.
We consider the series mass-spring-damper system shown in Figure~\ref{fig:msd} with parameters $m_1 = 8, m_2 = 6, k_1 = 5, k_2 = 4, c_1 = -0.5, c_2 = 1.3$, and sampling period $\varDelta T = 1$.

For the nominal MPC, we set $Q=I_4$ and $R=I_2$, and choose $Q_\mathrm{f}$ as the solution of DARE in \eqref{dare}.
The prediction horizon is $H = 2$.
The nominal closed-loop matrix satisfies $\rho(A_\mathrm{nom}) = 0.852$.

To encourage the states to satisfy $\underline{x} \leq x_i \leq \bar{x}$ with $\underline{x} = -0.6$ and $\bar{x} = 0.4$, we use the regularizer \eqref{Vreg:state_soft_penalty} from Example~\ref{example:state_soft_penalty}.
We choose the desired contraction factor $\eta = 0.99$.
Since $\eta>\rho(A_\mathrm{nom})$, Theorem~\ref{thm:lipschitz:feasibility_wellposedness} guarantees feasibility of the SDP that minimizes $\gamma=L^{-2}$ subject to \eqref{contractivity_lmi:lipschitz}.
One feasible solution gives $L = 0.60$, and hence Theorem~\ref{thm:contractivity:lipschitz} guarantees that the regularized MPC is strongly contracting with factor $\eta$ with respect to the norm $\|\cdot\|_P$, where $P$ is the LMI solution.
Since the regularized MPC satisfies $u=0$ at $x=0$, Lemma~\ref{lemma:contraction_benefit} implies that the closed-loop system remains exponentially stable.

To illustrate the certified contraction, we compute two trajectories $\{x^{(1)}(t)\}_{t=0, \dots, 20}$ and $\{x^{(2)}(t)\}_{t=0, \dots, 20}$ starting from
$
    x^{(1)}(0) = [-0.5;\, 3.5;\, -2.5;\, 1.5], \ \ 
    x^{(2)}(0) = [-2;\, -2;\, -2;\, -2]
    .
$
Figure~\ref{fig:full} shows the trajectories under the nominal MPC, the state-penalized MPC, and the open-loop system.
The top panel shows $\|x^{(1)}(t) - x^{(2)}(t)\|_P$ on a logarithmic scale.
For comparison, the exponential decay with rate $\eta$ is shown in gray.
Both MPC laws exhibit contraction with factor $\eta$, whereas the open-loop system does not.
The middle and bottom panels show the trajectories of the first state $x_1$ and the third state $x_3$, respectively, which correspond to the positions of the two masses.
Compared with the open-loop response, both MPC laws reduce the amplitude of the state oscillations.
The state-penalized MPC further encourages the states to remain within the interval $[-0.6,\,0.4]$.
In particular, $x_1$ enters this interval at time $4$ and remains inside, while $x_3$ enters it at time $6$ and remains inside.

\begin{figure}[t]
    \centering
    \includegraphics[width=0.99\linewidth]{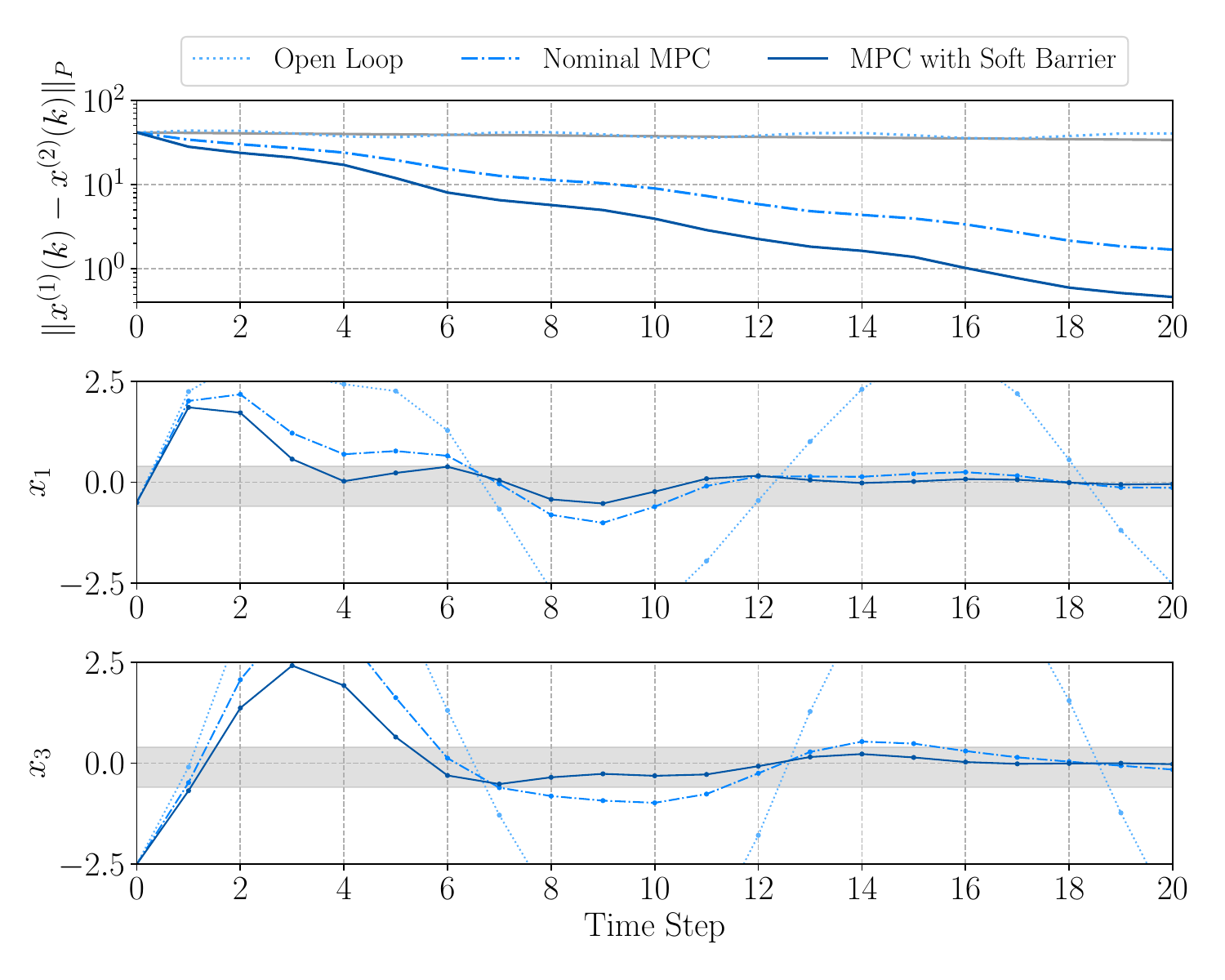}
    \caption{Soft state-penalty example.
    Top: distance between two trajectories, measured in the certified norm $\| \cdot \|_P$, with the gray line indicating the prescribed decay rate $\eta$.
    Middle and bottom: position trajectories $x_1$ and $x_3$ of the two masses.
    The state-penalized MPC preserves contraction and encourages the positions to remain in the desired interval $[-0.6,\,0.4]$, while the open-loop system does not exhibit the certified decay.
    \label{fig:full}
    }
\end{figure}

\section{Conclusion}
\label{sec:conclusion}

This paper investigated the contractivity of regularized MPC for linear systems through implicit Lur'e analysis.
The results provide a new perspective on stability analysis and design for regularized MPC.

Several directions for future research remain.
One promising application of the proposed contractivity analysis is economic MPC \cite{TF-LG-MAM:18}, because contractivity can ensure convergence to an unknown optimal operation.
Other potential applications include nonlinear MPC \cite{FB-ALB-MF-RS:24}, distributed MPC \cite{CC-CNJ-MM-MNZ:16}, and receding horizon games \cite{SH-GB-FD-DLMP:26}.
Furthermore, the implicit Lur'e approach may apply to a broader class of implicit problems, such as anti-windup design of PID controllers \cite{IQ-ST-GV-LZ:22}, moving horizon estimation-type observers \cite{JDS-SM-JK-MNZ-MAM:23}, and equilibrium networks \cite{EW-JZK:20, MR-RW-IRM:24} in machine learning.
It may also be extended to dynamical operator settings \cite{RW-IRM:19}, where the implicit equations are solved over time.
Finally, an extension of the framework to local stability analysis would enable the treatment of constrained MPC settings.

\section*{References}

\bibliographystyle{plainurl+isbn}

\bibliography{alias,Main,FB}

\appendices

\section{Supplemental Lemmata\\for Proof of Proposition~\ref{prop:contractivity_feasibility}}

\label{app:proof_feasibility}

\begin{lemma}[Core for LMI feasibility]
\label{lemma:feasibility_proof}
Let $N_1\in \R^{n\times n}, N_2 \in \R^{m\times m}$ be negative definite, let $B_{11} \in \R^{n\times n}, A_{22}, B_{22} \in \R^{m \times m}$ be symmetric, and let $A_{12}, B_{12} \in \R^{n\times m}$ be arbitrary matrices.
Then there exists a positive scalar $\alpha > 0$ such that
\begin{align}
    \label{scaling_multiplier_inequality}
    \begin{bmatrix}
        N_1 & A_{12} \\
        A_{12}^\top & A_{22}
    \end{bmatrix}
    +
    \alpha
    \begin{bmatrix}
        B_{11} & B_{12} \\
        B_{12}^\top & B_{22}
    \end{bmatrix}
    +
    \frac{1}{\alpha}
    \begin{bmatrix}
        0 & 0 \\
        0 & N_2
    \end{bmatrix}
    \prec 0
    .
\end{align}
\end{lemma}

\smallskip

\proof
Let $\bar{N}_2 \coloneqq N_2 + \alpha A_{22} + \alpha^2 B_{22}$.
Since $N_2 \prec 0$, we have $\bar{N}_2 \prec 0$ for a sufficiently small positive constant $\alpha$.
The Schur complement lemma implies that \eqref{scaling_multiplier_inequality} is equivalent to
\begin{align*}
    N_1 + \alpha B_{11} - \alpha (A_{12} + \alpha B_{12}) \bar{N}_2^{-1} (A_{12} + \alpha B_{12})^\top \prec 0
    ,
\end{align*}
which holds for a sufficiently small positive constant $\alpha$ since $N_1 \prec 0$.
\endproof

\begin{lemma}[Singleton property]
\label{lemma:singleton}
Let $\Psi: \R^r \to \mathcal{P}(\R^m)$ admit
\begin{align*}
    M =
    \begin{bmatrix}
        M_{11} & M_{12}
        \\
        M_{12}^\top & M_{22}
    \end{bmatrix}
\end{align*}
as an IMM.
If $M_{22} \prec 0$, then $\Psi(z)$ is a singleton for every $z \in \R^r$.
\end{lemma}

\proof
Fix $z \in \R^r$ and take $g^{(1)},g^{(2)} \in \Psi(z)$.
The IMM condition \eqref{incremental_multiplier_set_valued} with $z^{(1)}=z^{(2)}=z$ gives
\begin{align*}
    (g^{(1)}-g^{(2)})^\top M_{22} (g^{(1)}-g^{(2)}) \ge 0.
\end{align*}
Since $M_{22}\prec0$, this is possible only if $g^{(1)}=g^{(2)}$.
\endproof

\section{Proof of Proposition~\ref{prop:contractivity_feasibility}}
\label{appendix:proof_implicit_lure}

\proof
Let $g^{(j)} \in \Psi_t(z^{(j)})$ for $j=1,2$.
Since $\Psi_t = \alpha \Psi^\mathrm{base}_t$, we have $g^{(j)}/\alpha \in \Psi^\mathrm{base}_t(z^{(j)})$.
Thus, for each $i=1,\ldots,p$, the IMM condition for $\Psi^\mathrm{base}_t$ gives
\begin{align*}
    \begin{bmatrix}
        z^{(1)} - z^{(2)}
        \\
        g^{(1)} - g^{(2)}
    \end{bmatrix}^{\!\top}
    \begin{bmatrix}
        M_{11}^i & \frac{1}{\alpha} M_{12}^i
        \\
        \frac{1}{\alpha} (M_{12}^i)^\top & \frac{1}{\alpha^2} M_{22}^i
    \end{bmatrix}
    \begin{bmatrix}
        z^{(1)} - z^{(2)}
        \\
        g^{(1)} - g^{(2)}
    \end{bmatrix}
    \ge 0
    ,
\end{align*}
which implies $\Psi_t$ admits the IMM
$
    \begin{bmatrix}
        M_{11}^i & \frac{1}{\alpha} M_{12}^i
        \\
        \frac{1}{\alpha} (M_{12}^i)^\top & \frac{1}{\alpha^2} M_{22}^i
    \end{bmatrix}
$.

To prove the feasibility of \eqref{contractivity_lmi:feasibility}, denote $\subscr{A}{cl} \coloneqq A-BD^{-1}C$ and choose $\bar{\lambda}_1,\ldots,\bar{\lambda}_p$ as in the statement.
Let
\begin{align*}
    \bar{M}_{11}
    \coloneqq
    \sum_{i=1}^p \bar{\lambda}_i M_{11}^i,
    \
    \bar{M}_{12}
    \coloneqq
    \sum_{i=1}^p \bar{\lambda}_i M_{12}^i,
    \
    \bar{M}_{22}
    \coloneqq
    \sum_{i=1}^p \bar{\lambda}_i M_{22}^i.
\end{align*}
Since $D$ is invertible, $D^\top \bar{M}_{22} D \prec 0$.
Since $\rho(\subscr{A}{cl})<\eta$, we know $\rho(\subscr{A}{cl}/\eta)<1$ so that there exists $P_\mathrm{nom}\succ 0$ with $(\subscr{A}{cl}/\eta)^\top P_\mathrm{nom}(\subscr{A}{cl}/\eta) - P_\mathrm{nom}\prec 0$, that is, $\subscr{A}{cl}^\top P_\mathrm{nom}\subscr{A}{cl} - \eta^2 P_\mathrm{nom}\prec 0$.
Let
\begin{align*}
    T
    \coloneqq
    \begin{bmatrix}
        I_{\nstate} & 0
        \\
        -D^{-1} C & I_{\ninput}
    \end{bmatrix},
\end{align*}
which is nonsingular.
Set $P = P_\mathrm{nom}$ and $\lambda_i = \alpha \bar{\lambda}_i$ for $i=1,\dots,p$.
After applying the congruence transformation\footnote{
    which post-multiplies \eqref{contractivity_lmi:feasibility} by $T$ and pre-multiplies it by $T^\top$, yielding an equivalent inequality.
} by $T$, the inequality \eqref{contractivity_lmi:feasibility} can be written as
\begin{align}
    \label{contractivity_lmi:scaled_proof}
    \begin{aligned}
    \begin{bmatrix}
        N_1 & A_{12}
        \\
        A_{12}^\top & A_{22}
    \end{bmatrix}
    &+
    \alpha
    \begin{bmatrix}
        B_{11} & B_{12}
        \\
        B_{12}^\top & B_{22}
    \end{bmatrix}
    +
    \frac{1}{\alpha}
    \begin{bmatrix}
        0 & 0 \\
        0 & N_2
    \end{bmatrix}
    \prec 0
    ,
    \end{aligned}
\end{align}
where
$
    N_1
    \coloneqq
    \subscr{A}{cl}^\top P_\mathrm{nom}\subscr{A}{cl}
    -
    \eta^2 P_\mathrm{nom}
    \prec 0,
$
$
    N_2 \coloneqq D^\top \bar{M}_{22} D
    ,
$
and $A_{12}$, $A_{22}$, $B_{11}$, $B_{12}$, and $B_{22}$ are appropriate matrices of compatible dimensions, with $A_{22}$, $B_{11}$, and $B_{22}$ being symmetric.
Lemma~\ref{lemma:feasibility_proof} implies that there exists a positive scalar $\alpha > 0$ such that \eqref{contractivity_lmi:scaled_proof} is negative definite.
Since $T$ is nonsingular, the inequality \eqref{contractivity_lmi:feasibility} holds.
\endproof

\section{Proof of Well-Posedness}
\label{appendix:proof_well-posedness}

\subsection{Supplementary Lemma: Well-posedness via IMM}

\begin{lemma}[Well-posedness via IMM]
\label{lemma:well-posedness_imm}
Consider the generalized equation \eqref{generalized_equation}, whose $\Psi_t$ admits the IMMs $M_1, \dots, M_p \in \R^{(r+\ninput)\times(r+\ninput)}$.
Suppose $D$ is invertible and there exist nonnegative scalars $\bar{\lambda}_1, \dots, \bar{\lambda}_p$ such that $\bar{M}_{22} \prec 0$, where
\begin{align}
    \label{condition:IMM_summation_bar}
    \begin{bmatrix}
        \bar{M}_{11} & \bar{M}_{12} \\
        \bar{M}_{12}^\top & \bar{M}_{22}
    \end{bmatrix}
    \coloneqq
    \sum_{i=1}^p \bar{\lambda}_i M_i
\end{align}
and $\bar{M}_{11} \in \R^{r \times r}$, $\bar{M}_{12} \in \R^{r \times \ninput}$, and $\bar{M}_{22} \in \R^{\ninput \times \ninput}$.
Then, for each $t\in\timeset$ and each $z\in\R^r$, the set $\Psi_t(z)$ contains at most one element.
Moreover, suppose $\Psi_t(z)$ is nonempty for each $t$ and each $z$ and identify $\Psi_t$ as a function from $\R^r$ to $\R^\ninput$.
Then $\Psi_t$ is Lipschitz continuous with respect to $z$, uniformly in $t$.
Furthermore, if there exist a positive scalar $\mu$, a positive definite matrix $W$, and nonnegative scalars $\lambda_1, \dots, \lambda_p$ such that
\begin{align}
    \label{condition:strong_monotonicity}
    \begin{bmatrix}
        2(1-\mu) W & W \\ W & 0
    \end{bmatrix}
    \succeq
    \begin{bmatrix}
        G & 0 \\ 0 & D
    \end{bmatrix}^\top
    \Big(
        \sum_{i=1}^{p} \lambda_i
        M_i
    \Big)
    \begin{bmatrix}
        G & 0 \\ 0 & D
    \end{bmatrix}
    ,
\end{align}
then Assumption~\ref{assumption:generalized_equation} holds.
\end{lemma}

\proof
Fix $x \in \R^{\nstate}$ and $t \in \timeset$.
By summing up the IMM conditions for $\Psi_t$ defined by $M_1, \dots, M_p$ with weights $\bar{\lambda}_1, \dots, \bar{\lambda}_p$, we have
\begin{align}
    \label{incremental_multiplier_sum}
    \begin{bmatrix}
        z^{(1)} - z^{(2)} \\ g^{(1)} - g^{(2)}
    \end{bmatrix}^\top
    \begin{bmatrix}
        \bar{M}_{11} & \bar{M}_{12} \\ \bar{M}_{12}^\top & \bar{M}_{22}
    \end{bmatrix}
    \begin{bmatrix}
        z^{(1)} - z^{(2)} \\ g^{(1)} - g^{(2)}
    \end{bmatrix}
    \ge 0
\end{align}
for all $z^{(1)}, z^{(2)} \in \R^r$ and $g^{(1)} \in \Psi_t(z^{(1)})$, $g^{(2)} \in \Psi_t(z^{(2)})$.
By taking $z^{(1)}=z^{(2)}=z$, the IMM condition \eqref{incremental_multiplier_sum} implies
$
    (g^{(1)}-g^{(2)})^\top \bar M_{22} (g^{(1)}-g^{(2)}) \ge 0
$
for all $g^{(1)}, g^{(2)} \in \Psi_t(z)$, and hence $g^{(1)}=g^{(2)}$ since $\bar M_{22} \prec 0$.
Therefore, $\Psi_t(z)$ contains at most one element at each $z\in\R^r$.

To show the Lipschitzness, denote $\varDelta z \coloneqq z^{(1)}-z^{(2)}$, $\varDelta g \coloneqq \Psi_t(z^{(1)})-\Psi_t(z^{(2)})$, and $S \coloneqq \bar M_{11}-\bar M_{12}\bar M_{22}^{-1}\bar M_{12}^\top$.
The IMM condition \eqref{incremental_multiplier_sum} can be rewritten as
\begin{align*}
    \varDelta z^\top
    S
    \varDelta z
    \geq
    \| \varDelta g + \bar M_{22}^{-1}\bar M_{12}^\top \varDelta z \|_{- \bar M_{22}}^2
    ,
\end{align*}
which in turn yields
\begin{align}
    \label{inequality:IMM_norm}
    \rho(S) \|\varDelta z\|^2
    \ge
    \lambda_{\min}(-\bar M_{22})
    \|
    \varDelta g + \bar M_{22}^{-1}\bar M_{12}^\top \varDelta z
    \|^2
    ,
\end{align}
where $\lambda_{\min}(\cdot)$ denotes the minimum eigenvalue of a matrix.
Since $-\bar M_{22} \succ 0$, we have $\lambda_{\min}(-\bar M_{22}) > 0$.
Using the triangle inequality and \eqref{inequality:IMM_norm}, we obtain
\begin{align*}
    \|\varDelta g\|
    &\le \|\varDelta g + \bar M_{22}^{-1}\bar M_{12}^\top \varDelta z\| + \|\bar M_{22}^{-1}\bar M_{12}^\top \varDelta z\|
    \\
    &\le \left(\sqrt{\frac{\rho(S)}{\lambda_{\min}(-\bar M_{22})}} + \|\bar M_{22}^{-1}\bar M_{12}^\top\|\right) \|\varDelta z\|
    .
\end{align*}
This establishes the uniform Lipschitzness of $\Psi_t$.

We conclude the proof by showing the mapping
\begin{align}
    \label{residual_mapping}
    \mathcal{R}_{x,t}(u)
    \coloneqq
    u + D^{-1} \Psi_t(Fx + G u)
\end{align}
is strongly monotone, which implies that $\mathcal{R}_{x,t}$ is maximally monotone \cite[p.11]{EKR-SB:16} and the equation $\mathcal{R}_{x,t}(u)=-D^{-1}Cx$ (which is equivalent to \eqref{generalized_equation}) admits a unique solution \cite[p.20]{EKR-SB:16}.
Let $u^{(1)}, u^{(2)} \in \R^{\ninput}$ be arbitrary and let $z^{(1)} = Fx + Gu^{(1)}$ and $z^{(2)} = Fx + Gu^{(2)}$.
Let 
\begin{align*}
    Y \coloneqq
    \begin{bmatrix}
        u^{(1)} - u^{(2)} \\ D^{-1} \Psi_t(z^{(1)}) - D^{-1} \Psi_t(z^{(2)})
    \end{bmatrix}
    .
\end{align*}
By substituting $z^{(1)} = Fx + Gu^{(1)}$, $z^{(2)} = Fx + Gu^{(2)}$, $g^{(1)} = \Psi_t(z^{(1)})$, and $g^{(2)} = \Psi_t(z^{(2)})$ into \eqref{incremental_multiplier_sum}, we obtain
\begin{align}
    \label{incremental_multiplier_sum:substitution}
    Y^\top
    \begin{bmatrix}
        G & 0 \\ 0 & D
    \end{bmatrix}^\top
    \Big(
    \sum_{i=1}^p \lambda_i M_i
    \Big)
    \begin{bmatrix}
        G & 0 \\ 0 & D
    \end{bmatrix}
    Y
    \ge 0
    .
\end{align}
Upon pre-multiplying \eqref{condition:strong_monotonicity} by $Y^\top$, post-multiplying by $Y$, and taking into account \eqref{incremental_multiplier_sum:substitution}, we have
\begin{align*}
    Y^\top
    \begin{bmatrix}
        2(1-\mu) W & W \\ W & 0
    \end{bmatrix}
    Y
    \geq 0
    ,
\end{align*}
which states that the mapping $\mathcal{R}_{x,t}$ in \eqref{residual_mapping} is $\mu$-strongly monotone with respect to metric $W$, i.e.,
$
        (u^{(1)} - u^{(2)})^\top
        W
        (\mathcal{R}_{x,t}(u^{(1)}) - \mathcal{R}_{x,t}(u^{(2)}))
        \ge \mu \|u^{(1)} - u^{(2)}\|_W^2
        .
$
\endproof

\subsection{Proof of Proposition~\ref{prop:well-posedness_contraction}}

\proof
Since $\bar{M}_{12} = 0$, the lower-right part of the strict version of \eqref{contractivity_lmi_implicit} reads
\begin{align*}
    B^\top P B + G^\top \bar{M}_{11} G + D^\top \bar{M}_{22} D
    \prec 0
    .
\end{align*}
Since $B^\top P B \succeq 0$, we have $G^\top \bar{M}_{11} G + D^\top \bar{M}_{22} D \prec 0$.

We show that \eqref{condition:strong_monotonicity} holds with $W = - D^\top \bar{M}_{22} D$.
Note that $- D^\top \bar{M}_{22} D \succ 0$ since $\bar{M}_{22} \prec 0$ and $D$ is invertible.
Since $G^\top \bar{M}_{11} G \prec - D^\top \bar{M}_{22} D$, there exists $\mu \in (0,\frac12)$ such that $G^\top \bar{M}_{11} G \preceq - (1-2\mu) D^\top \bar{M}_{22} D$.
Thus, by denoting $W = - D^\top \bar{M}_{22} D$, we have
\begin{align*}
    \begin{bmatrix}
        (1-2\mu) W - G^\top \bar{M}_{11} G & 0 \\ 0 & 0
    \end{bmatrix}
    +
    \begin{bmatrix}
        W & W \\ W & W
    \end{bmatrix}
    \succeq 0
    ,
\end{align*}
which is exactly \eqref{condition:strong_monotonicity} with $\lambda_i = \bar{\lambda}_i$ for each $i=1,\ldots,p$.
Therefore, Lemma~\ref{lemma:well-posedness_imm} implies that Assumption~\ref{assumption:generalized_equation} holds.
\endproof

\section{Multipliers for Repeated Perturbations}

\begin{lemma}
\label{lemma:multiplier_repeated_structure}
Let $\Psi_t: \R^{H\ninput} \to \mathcal{P}(\R^{H\ninput})$ be defined by
\begin{align}
    \label{psi:repeated_structure}
    \Psi_t(U) =
    \Phi_{t+1}(u_1) \oplus \cdots \oplus \Phi_{t+H}(u_H)
    ,
\end{align}
where $U=[u_1; \ldots; u_H]$ and $\Phi_t: \R^{\ninput} \to \mathcal{P}(\R^{\ninput})$.
Suppose that $\Phi_t$ admits the IMM $M \in \R^{2\ninput \times 2\ninput}$ for each $t \in \timeset$.
Then $\Psi_t$ admits the IMMs $\hat{M}^1,\dots, \hat{M}^H$, where
\begin{align}
    \label{multiplier:extended}
    \hat{M}^h =
    \begin{bmatrix}
        \Pi_h & 0 \\ 0 & \Pi_h
    \end{bmatrix}^\top
    M
    \begin{bmatrix}
        \Pi_h & 0 \\ 0 & \Pi_h
    \end{bmatrix}
    ,
\end{align}
where the $(k,h\ninput-h+k)$-th element of $\Pi_h \in \R^{\ninput \times H\ninput}$ is $1$ for each $k=1,\dots,\ninput$ and all other elements are $0$.\footnote{
    The matrix $\Pi_1$ in Lemma~\ref{lemma:multiplier_repeated_structure} is the same as that in \S\ref{sec:problem_setup}.
}
\end{lemma}

\proof
Let $U^{(1)}, U^{(2)} \in \R^{H\ninput}$, $G^{(1)} \in \Psi_t(U^{(1)})$, and $G^{(2)} \in \Psi_t(U^{(2)})$ be arbitrary vectors.
Define $u^{(i)}_h \coloneqq \Pi_h U^{(i)}$ and $g^{(i)}_h \coloneqq \Pi_h G^{(i)}$ for $i=1,2$ and $h=1,\dots,H$, i.e.,
\begin{align}
    \label{projection_ug}
    \begin{bmatrix}
    u_h^{(1)} - u_h^{(2)} \\ g_h^{(1)} - g_h^{(2)}
    \end{bmatrix}
    =
    \begin{bmatrix}
        \Pi_h & 0 \\ 0 & \Pi_h
    \end{bmatrix}
    \begin{bmatrix}
        U^{(1)}- U^{(2)} \\ G^{(1)} - G^{(2)}
    \end{bmatrix}
    .
\end{align}
Note that $g^{(i)}_h \in \Phi_{t+h}(u^{(i)}_h)$ for each $i=1,2$ and $h=1,\dots,H$.
Since $\Phi_t$ admits the IMM $M$ for each $t$, we have
\begin{align}
    \label{multiplier_condition:phi}
    \begin{bmatrix}
    u^{(1)}_h - u^{(2)}_h \\ g^{(1)}_h - g^{(2)}_h
    \end{bmatrix}^\top
    M
    \begin{bmatrix}
    u^{(1)}_h - u^{(2)}_h \\ g^{(1)}_h - g^{(2)}_h
    \end{bmatrix}
    \ge 0
\end{align}
for the above $u^{(1)}_h$, $u^{(2)}_h$, $g^{(1)}_h$, and $g^{(2)}_h$.
By substituting \eqref{projection_ug} into \eqref{multiplier_condition:phi}, we have the IMM condition for $\Psi_t$ with $\hat{M}^h$.
\endproof

\section{Proof of Theorems \ref{thm:contractivity_lmi:convex_smooth}--\ref{thm:lipschitz:feasibility_wellposedness}}

In the setup of the regularized MPC, the plant dynamics \eqref{plant_lti} can be reformulated as
\begin{align}
    \label{plant_lti_extended}
    x(t+1) = A x(t) + B \Pi_1 U(t)
    .
\end{align}

\subsection{Proof of Theorem~\ref{thm:contractivity_lmi:convex_smooth}}
\label{appendix:proof_convex_smooth}

\proof
The closed-loop dynamics of the regularized MPC in Case~1 can be represented as the implicit Lur'e system consisting of the plant \eqref{plant_lti_extended} and the generalized equation
\begin{align}
    \label{optimality_condition:convex_smooth}
    Cx(t) + DU(t) + \psi_t(U(t)) = 0
    ,
\end{align}
where \eqref{optimality_condition:convex_smooth} is the first-order optimality condition and
\begin{align}
    \label{psi:convex_smooth}
    \psi_t(U)
    \coloneqq
    \begin{bmatrix} \nabla_u \ellreg(u_1, t+1) \\ \vdots \\ \nabla_u \ellreg(u_H, t+H) \end{bmatrix}
    .
\end{align}
Lemma~\ref{lemma:wellposedness_subdifferential} implies that Assumption~\ref{assumption:generalized_equation} holds for \eqref{optimality_condition:convex_smooth}.
It follows from Example~\ref{example:multiplier_convex_analysis} that $\nabla_u \ellreg^\mathrm{base}(\cdot, t)$ admits $M_\mathrm{cnv} \coloneqq \begin{bmatrix} 0 & I_{\ninput} \\ I_{\ninput} & -\frac{2}{L} I_{\ninput} \end{bmatrix}$ as an IMM.
By Lemma~\ref{lemma:multiplier_repeated_structure}, $\psi_t$ in \eqref{psi:convex_smooth} admits the IMM
$\hat{M}^1, \dots, \hat{M}^H$ in the form of \eqref{multiplier:extended} with $M = M_\mathrm{cnv}$.
Straightforward calculations show that
\begin{align*}
    \sum_{h=1}^H \lambda_h \hat{M}^h
    =
    \begin{bmatrix}
        0 & \Lambda \\ \Lambda & -\frac{2}{L} \Lambda
    \end{bmatrix}
    .
\end{align*}
Therefore, by Corollary~\ref{cor:implicit_lure_contraction} with $F=0$ and $G=I_{H\ninput}$, \eqref{contractivity_lmi:convex_smooth} implies the contractivity of the closed-loop system.

Positive scalars $\lambda_1, \dots, \lambda_H$ provide $\bar{M}_{22} = -\frac{2}{L}\Lambda \prec 0$.
Proposition~\ref{prop:contractivity_feasibility} implies that, if the input is determined by the solution of the generalized equation
$
    Cx(t) + DU(t) + \alpha \psi_t(U(t)) = 0
$
with a positive scalar $\alpha$, then there exist a positive definite matrix $P$, nonnegative scalars $\lambda_1, \dots, \lambda_H$, and a positive scalar $\alpha$ satisfying
\begin{align*}
    \begin{aligned}
    &
    \begin{bmatrix}
        A^\top P A - \eta^2 P & A^\top P B \Pi_1 \\
        \Pi_1^\top B^\top P A & \Pi_1^\top B^\top P B \Pi_1
    \end{bmatrix}
    \\
    &
    +
    \begin{bmatrix}
        0 & I_{H\ninput} \\ -C & -D
    \end{bmatrix}^\top
    \begin{bmatrix}
        0
        &
        \frac{1}{\alpha} \Lambda
        \\
        \frac{1}{\alpha} \Lambda
        &
        - \frac{2}{L \alpha^2} \Lambda
    \end{bmatrix}
    \begin{bmatrix}
        0 & I_{H\ninput} \\ -C & -D
    \end{bmatrix}
    \prec 0
    .
    \end{aligned}
\end{align*}
Therefore, \eqref{contractivity_lmi:convex_smooth} holds with $L \alpha$ in place of $L$ and $\alpha \lambda_h$ in place of $\lambda_h$ for each $h=1,\dots,H$.
\endproof

\subsection{Proof of Theorem~\ref{thm:contractivity_lmi:ccp}}
\label{appendix:proof_ccp}

\proof
The closed-loop dynamics of the regularized MPC in Case~2 can be represented as the implicit Lur'e system consisting of the plant \eqref{plant_lti_extended} and the generalized equation
\begin{align}
    \label{optimality_condition:ccp}
    0 \in Cx(t) + DU(t) + \Psi_t(U(t))
    ,
\end{align}
where \eqref{optimality_condition:ccp} is the first-order optimality condition and
\begin{align}
    \label{psi:ccp}
    \Psi_t(U)
    \coloneqq
    \partial_u \ellreg(u_1, t+1)
    \oplus \cdots \oplus
    \partial_u \ellreg(u_H, t+H)
    .
\end{align}
Lemma~\ref{lemma:wellposedness_subdifferential} implies that Assumption~\ref{assumption:generalized_equation} holds for the generalized equation \eqref{optimality_condition:ccp}.
Since $\ellreg(\cdot, t)$ is CCP, the subdifferential $\partial_u \ellreg(\cdot, t)$ is monotone \cite[p.10]{EKR-SB:16}.
According to Example~\ref{example:multiplier_monotone}, $\partial_u \ellreg(\cdot, t)$ admits the IMM $M_\mathrm{ccp} \coloneqq \begin{bmatrix} 0 & I_{\ninput} \\ I_{\ninput} & 0 \end{bmatrix}$ for each $t \in \timeset$.
By Lemma~\ref{lemma:multiplier_repeated_structure}, $\Psi_t$ in \eqref{psi:ccp} admits $\hat{M}^1, \dots, \hat{M}^H$ as IMMs in the form of \eqref{multiplier:extended} with $M = M_\mathrm{ccp}$.
Straightforward calculations show that
\begin{align*}
    \sum_{h=1}^H \lambda_h \hat{M}^h
    =
    \begin{bmatrix}
        0 & \Lambda \\ \Lambda & 0
    \end{bmatrix}
    .
\end{align*}
Therefore, by Corollary~\ref{cor:implicit_lure_contraction} with $F=0$ and $G=I_{H\ninput}$, \eqref{contractivity_lmi:ccp} implies the contractivity of the closed-loop system.
\endproof

\subsection{Proof of Theorem~\ref{thm:contractivity:lipschitz}}
\label{appendix:proof_lipschitz}

\proof
The closed-loop dynamics of the regularized MPC in Case~3 can be formulated as \eqref{plant_lti_extended} with \eqref{equation:implicit}.
Let $r \coloneqq H(\nstate+\ninput)$, $z \coloneqq [X;U] \in \R^{r}$,
$
    F \coloneqq
    \begin{bmatrix}
        \bar{A} \\ 0
    \end{bmatrix}
    \in \R^{r \times \nstate}
$, and
$
    G \coloneqq
    \begin{bmatrix}
        \bar{B} \\ I_{H\ninput}
    \end{bmatrix}
    \in \R^{r \times H\ninput}
$.
Since $\psi_t$ is Lipschitz with Lipschitz constant $L$, $\psi_t$ admits $M_\mathrm{lip} = \diag{I_{r}, -\frac{1}{L^2} I_{H\ninput}}$ as an IMM according to Example~\ref{example:multiplier_lipschitz}.
Since \eqref{equation:implicit} is the first-order optimality condition of the OCP \eqref{ocp}, Assumption~\ref{assumption:lipschitz:well-posedness} implies that Assumption~\ref{assumption:generalized_equation} holds for the generalized equation \eqref{equation:implicit}.
By Corollary~\ref{cor:implicit_lure_contraction} with $p=1$ and $\lambda_1=1$, we see that \eqref{contractivity_lmi:lipschitz} implies the contractivity.
\endproof

\subsection{Proof of Theorem~\ref{thm:lipschitz:feasibility_wellposedness}}
\label{appendix:proof_lipschitz:feasibility_wellposedness}

\proof
Let $r \coloneqq H(\nstate+\ninput)$, $z \coloneqq [X;U] \in \R^{r}$,
$
    F \coloneqq
    \begin{bmatrix}
        \bar{A} \\ 0
    \end{bmatrix}
    \in \R^{r \times \nstate}
$, and
$
    G \coloneqq
    \begin{bmatrix}
        \bar{B} \\ I_{H\ninput}
    \end{bmatrix}
    \in \R^{r \times H\ninput}
$.
Since $\psi_t$ in \eqref{equation:implicit} is Lipschitz with Lipschitz constant $L$, $\psi_t$ admits the IMM $M_\mathrm{lip} \coloneqq \diag{I_{r}, -\frac{1}{L^2} I_{H\ninput}}$ according to Example~\ref{example:multiplier_lipschitz}.
Since $\bar{M}_{22} = -I_{H\ninput} \prec 0$ with $\bar{\lambda}_1 = 1$, Proposition~\ref{prop:contractivity_feasibility} implies that, if the input is determined by the solution of the generalized equation
$
    Cx + DU + \alpha \psi_t(Fx + GU) = 0
$
with a positive scalar $\alpha$, then there exist a positive definite matrix $P$, nonnegative scalars $\lambda_1, \dots, \lambda_H$, and a positive scalar $\alpha$ such that
\begin{align*}
    \begin{aligned}
    &
    \begin{bmatrix}
        A^\top P A - \eta^2 P & A^\top P B \Pi_1 \\
        \Pi_1^\top B^\top P A & \Pi_1^\top B^\top P B \Pi_1
    \end{bmatrix}
    \\
    &
    +
    \begin{bmatrix}
        \bar{A} & \bar{B}
        \\
        0 & I_{H\ninput}
        \\
        -C & -D
    \end{bmatrix}^\top
    \begin{bmatrix}
    \lambda_1 I_{r} & 0
    \\
    0 & - \frac{\lambda_1}{L^2 \alpha^2} I_{H\ninput}
    \end{bmatrix}
    \begin{bmatrix}
        \bar{A} & \bar{B}
        \\
        0 & I_{H\ninput}
        \\
        -C & -D
    \end{bmatrix}
    \preceq 0
    .
    \end{aligned}
\end{align*}
Here, $\lambda_1 \neq 0$; otherwise, the lower-right block cannot be negative definite.
Therefore, the strict version of \eqref{contractivity_lmi:lipschitz} holds with $P/\lambda_1$ in place of $P$ and $L^{-2} \alpha^{-2}$ in place of $\gamma$.
Since $\bar{M}_{12} = 0$ and $\bar{M}_{22} = -\lambda_1 I_{H\ninput} \prec 0$, Proposition~\ref{prop:well-posedness_contraction} implies that the generalized equation is well-posed, which in turn implies that Assumption~\ref{assumption:lipschitz:well-posedness} holds.
\endproof

\section{Details on Comparison in Remark \ref{remark:novelty:ccp}}
\label{appendix:comparison_multipliers}

One comparable result to Theorem~\ref{thm:contractivity_lmi:ccp} is \cite[Theorem 1]{SH-GB-FD-DLMP:26}.
The paper \cite{SH-GB-FD-DLMP:26} treats game theoretic control of a linear time-invariant state-space system, where each agent minimizes their own costs under interactive input constraints.
Each optimization problem of the agents is assumed to be $\mu$-strongly convex, and the input constraints are assumed to be convex, closed, and nonempty.
The main result \cite[Theorem 1]{SH-GB-FD-DLMP:26} shows that it admits a globally asymptotically stable Nash equilibrium if a dissipativity-based LMI holds.

The generalized Nash equilibrium is defined as
\begin{align}
    \label{gne}
    0 \in F_u (U) + C x + N_\mathcal{U}(U)
    ,
\end{align}
where $F_u$ is a $\mu$-strongly monotone function, $C$ is a matrix, and $N_\mathcal{U}$ is the normal cone of $\mathcal{U}$ at $U$.
Note that $N_\mathcal{U}$ is monotone since $\mathcal{U}$ is convex, closed, and nonempty.
The core observation that leads to the dissipativity-based LMI in \cite[Theorem 1]{SH-GB-FD-DLMP:26} is \cite[Proposition 2]{SH-GB-FD-DLMP:26}, which states that the solution mapping of \eqref{gne} admits the IMM
\begin{align}
    \label{multiplier:rhg}
    M_\mathrm{gne}
    =
    \begin{bmatrix}
        0 & -C^\top \\ - C & - 2 \mu I_{H\ninput}
    \end{bmatrix}
    .
\end{align}

What we would like to highlight here is that the matrix inequality \eqref{contractivity_lmi:ccp} in the proposed framework is less conservative than the LMI in \cite[Theorem 1]{SH-GB-FD-DLMP:26} in two respects.
First, the proposed implicit Lur'e analysis in \S\ref{section:contractivity:ccp} exploits the matrix $D$ explicitly.
The $\mu$-monotonicity of the objective function falls into a special case of the proposed analysis, where $D$ is chosen as $D = \mu I$.
The other thing is that Theorem~\ref{thm:contractivity_lmi:ccp} allows $\lambda$s, since in our setting the regularizer is assumed to be a sum of stage-wise regularizing costs.

\end{document}